\documentclass[11pt,reqno]{amsart}
\usepackage{geometry}
\usepackage{enumitem}
\usepackage{amsmath}
\usepackage{amssymb}
\usepackage{amsthm}
\usepackage{amsrefs}
\usepackage{amsfonts}
\usepackage[dvipsnames]{xcolor}
\usepackage{mathtools}
\usepackage{mathrsfs}
\usepackage{hyperref}

\makeatletter
\@namedef{subjclassname@2020}{%
  \textup{2020} MSC}
\makeatother

\hypersetup{
colorlinks   = true, 
urlcolor     = blue, 
linkcolor    = purple, 
citecolor   = blue 
}

\def\XXint#1#2#3{{\setbox0=\hbox{$#1{#2#3}{\int}$ }
\vcenter{\hbox{$#2#3$ }}\kern-.6\wd0}}

\makeatletter
\newcommand*{\rom}[1]{\expandafter\@slowromancap\romannumeral #1@}

\newcommand{\ind}{\protect\raisebox{2pt}{$\mathbf{1}$}}

\newcommand{\SL}{\text{SL}}

\newcommand{\X}{\mathcal{X}}

\newcommand{\R}{\mathbb{R}}

\newcommand{\Z}{\mathbb{Z}}

\newcommand{\N}{\mathbb{N}}

\newcommand{\bthm}{\begin{thm}}
\newcommand{\ethm}{\end{thm}}
\newcommand{\bproof}{\begin{proof}}
\newcommand{\eproof}{\end{proof}}
\newcommand{\blem}{\begin{lem}}
\newcommand{\elem}{\end{lem}}
\newcommand{\brem}{\begin{rem}}
\newcommand{\erem}{\end{rem}}
\newcommand{\eeqn}{\end{equation}}
\newcommand{\eeqnn}{\end{equation*}}
\newcommand{\beqn}{\begin{equation}}
\newcommand{\beqnn}{\begin{equation*}}
\newcommand{\eprop}{\end{prop}}
\newcommand{\eexm}{\end{exm}}
\newcommand{\enexm}{\end{nexm}}
\newcommand{\ecor}{\end{cor}}
\newcommand{\bcor}{\begin{cor}}
\newcommand{\bexm}{\begin{exm}}
\newcommand{\bnexm}{\begin{nexm}}
\newcommand{\bprop}{\begin{prop}}
\newcommand{\bdefn}{\begin{defn}}
\newcommand{\edefn}{\end{defn}}
\newcommand{\benum}{\begin{enumerate}}
\newcommand{\eenum}{\end{enumerate}}

\newcommand{\bfv}{\mathbf{v}}
\newcommand{\bfw}{\mathbf{w}}
\newcommand{\bfe}{\mathbf{e}}
\newcommand{\cN}{\mathcal{N}}

\newcommand{\supp}{\text{supp}}

\newcommand{\id}{\mathbf{1}}
\DeclareMathOperator{\inj}{inj}

\newcommand{\rS}{\mathscr{S}}
\newcommand{\an}{\varkappa}
\newcommand{\Lie}{\mathrm{Lie}}

\title[Cone upgrade and Khintchine--Schmidt on fractals]{Cone upgrade for effective equidistribution and weighted Khintchine--Schmidt theorem on fractals}

\begin{document}
\theoremstyle{plain}
\newtheorem{thm}{Theorem}[section]
\newtheorem{lem}[thm]{Lemma}
\newtheorem{prop}[thm]{Proposition}
\newtheorem{cor}[thm]{Corollary}
\newtheorem{question}{Question}
\newtheorem{con}{Conjecture}
\theoremstyle{definition}
\newtheorem{defn}[thm]{Definition}
\newtheorem{exm}[thm]{Example}
\newtheorem{nexm}[thm]{Non Example}
\newtheorem{prob}[thm]{Problem}

\theoremstyle{remark}
\newtheorem{rem}[thm]{Remark}

\author{Gaurav Aggarwal}
\address{\textbf{Gaurav Aggarwal} \\
Institut f\"ur Mathematik, Universit\"at Z\"urich, 8057 Z\"urich, Switzerland}
\email{gaurav.aggarwal@math.uzh.ch}

\author{Alexander Gorodnik}
\address{\textbf{Alexander Gorodnik} \\
Institut f\"ur Mathematik, Universit\"at Z\"urich, 8057 Z\"urich, Switzerland}
\email{alexander.gorodnik@math.uzh.ch}

\date{}

\thanks{ G.~A. and A.~G. gratefully acknowledge support from the SNF grant 200020--212617.}



\begin{abstract}
We prove a fractal analogue of the weighted theorems of Khintchine and Schmidt for
products of self-similar measures on the real line whose defining iterated
function systems share a common contraction ratio. The proof
translates the counting problem into homogeneous dynamics and builds on recent
work of B\'enard, He and Zhang, who treated the unweighted case. Our main new
ingredient is a \emph{cone upgrade}: an effective double equidistribution theorem for translates of
product self-similar measures, uniform over diagonal elements whose weights
range over a fixed cone in the positive Weyl chamber, deduced from effective
equidistribution for a single weight. This yields a partial
effective form of a theorem of Khalil, Luethi and Weiss.
\end{abstract}

\maketitle

\tableofcontents

\section{Introduction}

Let $\Psi=(\psi_1,\ldots,\psi_m)$ be an $m$-tuple of non-increasing
functions and, for $\theta\in\R^m$, consider the system of inequalities
\begin{align}
\label{eq:intro system}
    |q\theta_i+p_i|\leq\psi_i(q)
    \qquad(1\leq i\leq m),
    \qquad (p,q)\in\Z^m\times\N .
\end{align}
With respect to Lebesgue measure on $\R^m$ the question is classical, and it
was answered from the outset in the \emph{weighted} form, in which the $\psi_i$
are allowed to differ and the quality of approximation is prescribed
separately in each coordinate. Writing $\psi=\psi_1\cdots\psi_m$, Khintchine's
theorem asserts that almost every $\theta$ admits infinitely many solutions
of~\eqref{eq:intro system} if $\sum_q\psi(q)$ diverges, and that almost no
$\theta$ does if it converges. In the
divergent case Schmidt~\cite[Thm.~1]{S60} went further and obtained an
asymptotic formula, with an error term, for the number of solutions with
$q\leq T$. What survives of this picture when Lebesgue measure is replaced by
a measure carried by a fractal is a much more recent story, and one in which
the weighted problem has so far remained out of reach.

Mahler~\cite{Mahler} asked whether the points of the middle-thirds Cantor set
behave, from the point of view of Diophantine approximation, like typical real
numbers, and Kleinbock, Lindenstrauss and
Weiss~\cite[Question~10.1]{KLW} recast his question
by asking whether measures supported on fractals satisfy a Khintchine
dichotomy. The two halves of the dichotomy have been approached
separately, and at first only for particular approximating functions. On the
convergence side, the scale that has received most attention is
$\psi(q)=q^{-1-\varepsilon}$, where the assertion is that almost no point of the
fractal is very well approximable. This was proved by Weiss~\cite{Weiss2001}
for a class of measures on the line including the middle-thirds Cantor measure,
and extended in~\cite{KLW} to all \emph{friendly} measures on $\R^d$, in the
stronger multiplicative form, which already covers every weight at this scale;
see also~\cites{PollingtonVelani, DFSUa, DFSUb}. On the divergence side, the
borderline function $\psi(q)=\varepsilon/q$ was treated by Einsiedler, Fishman
and Shapira~\cite{EinsiedlerFishmanShapira} for missing-digit Cantor measures,
and then, for irreducible self-similar measures, by Simmons and
Weiss~\cite{SimmonsWeiss}, who showed that for every $\varepsilon>0$ almost
every point admits infinitely many solutions; see
also~\cites{Khalilsing, dattajana24, AG24Random, AG24Levy, AG25b, AG25}.

For a general $\psi$, Khintchine's theorem on fractals was proved by Khalil
and Luethi~\cite{KhalilLuethi} for self-similar measures of large enough
dimension arising from a rational, finitely supported, contractive iterated
function system satisfying the open set condition, and then in full generality
on $\R$ by B\'enard, He and Zhang~\cite{benard24} and, on $\R^d$, by the same
authors~\cite{benard2026}, who also obtained Schmidt's counting asymptotic.
All of these results are unweighted. Weighted statements of a more restricted
kind are due to Prohaska, Sert and Shi~\cite{ProhaskaSertShi}, who could
however only treat measures defined by an iterated function system depending on
the weight, and to~\cite[Thm.~1.2]{KhalilWeiss}, where the set of weighted
Dirichlet improvable vectors on a carpet is shown to be null.
What lies behind these results is the equidistribution of
the diagonal translates of the fractal measure in the space of lattices. Both
\cite{KhalilLuethi} and \cites{benard24, benard2026} establish this effectively,
but only for a single diagonal flow, and it is this restriction to one flow that
confines their Diophantine consequences to the unweighted case. In the opposite
direction, Khalil, Luethi and Weiss~\cite{KhalilWeiss} proved equidistribution
along arbitrary diagonal sequences, hence for all weights at once, but only
qualitatively and only for carpets; they explicitly ask for an effective
version, noting that it would have important Diophantine applications.

This paper supplies such an effective version for products of self-similar
measures with a common contraction ratio, uniformly over weights in a compact
subcone of the positive Weyl chamber (Theorem~\ref{thm:Double equi}), and
deduces from it the weighted Khintchine--Schmidt theorem
(Theorem~\ref{main thm:counting}).

The mechanism is the Dani correspondence of Kleinbock and
Margulis~\cites{KM98,KM3}: solutions of~\eqref{eq:intro system} with $q$ of a
given size correspond to short vectors in the lattice $g_t(\bfw)u(\theta)\Z^{m+1}$,
where the diagonal element $g_t(\bfw)$ is determined by $\Psi$, and the counting
problem becomes one about the distribution of the translated measure
$g_t(\bfw)_*\bigl(u(\cdot)\Gamma\bigr)_*\nu$ in the space of unimodular lattices.
Two features distinguish the weighted problem. First, the weight $\bfw$ varies
with $t$, so that one needs equidistribution not for a single flow but uniformly
over a family of them; the hypothesis that $\Psi$ be $\alpha$-balanced
(Definition~\ref{defn:balanced}) is exactly what confines this family to a
compact subcone of the positive Weyl chamber. Second, the Borel--Cantelli
argument in the divergent case requires control of the correlations between two
different times, and hence a \emph{double} equidistribution statement, whose
error term must be governed by the separation of the two diagonal elements; this
is what the $\beta$-proportionality hypothesis provides. The remaining
ingredients --- quantitative non-divergence, moment estimates for Siegel
transforms, and the passage from an effective equidistribution statement to a
pointwise count --- follow the by now standard scheme of~\cites{KM98, KSW,
benard2026}.

The method of \cite{benard2026} does not apply directly. It rests on a
bootstrap for the dimension of the associated random walk at a single scale,
which is available because conjugation by the equal-weight flow acts on the
unipotent group by a homothety and therefore preserves the self-similar
structure of the measure. A general weight expands the coordinates of the
unipotent group at different rates, so that the single scale is replaced by
several, and the unipotent group is then no longer the full expanding
horospherical subgroup. We do not extend that method. Instead we factor a
diagonal element of arbitrary weight into an equal-weight element and a second
element of arbitrary weight, and use the self-similar decomposition of the
measure to absorb the second factor into the base point, so that the
equidistribution theorem of \cite{benard2026} has to be applied only to the
equal-weight factor; this factorisation is the mechanism behind the cone
upgrade, and it is carried out in
Proposition~\ref{prop:change of weights}. It is for this reason that the
theorem is needed at arbitrary base points, and this is also where
$\alpha$-balance is essential: it keeps the equal-weight part of the flow a
fixed proportion of the whole.

\medskip

The paper is organised as follows. Section~\ref{sec:notation} fixes notation and
records the decomposition of a self-similar measure that will be used
throughout, together with the mixing estimate of Theorem~\ref{thm:mixing}.
Section~\ref{sec:non div} constructs the height function and establishes the
non-divergence estimates, uniformly over weights in $[\alpha,1]^m$.
Section~\ref{sec:Double equi} contains the cone upgrade, that is, the proof of
Theorem~\ref{thm:Double equi}.
Sections~\ref{sec:con case} and~\ref{sec:div case} treat the convergence and the
divergence case of Theorem~\ref{main thm:counting} respectively, and
Section~\ref{sec:proof of main} completes its proof by verifying the hypotheses
of Theorem~\ref{thm:Double equi} for the measures at hand.

Throughout the paper, we fix an integer $m \in \N$ with $m \geq 2$.

  \begin{defn}
\label{defn:ss}
A probability measure $\sigma$ on $\R$ will be called a \emph{self-similar measure with equal contraction ratio} if there exist $c\in(0,1)$, a finite set $E$ and weights $\lambda_e>0$ $(e\in E)$ with
\[
\sum_{e\in E} \lambda_e=1,
\]
and affine contractions
\[
\phi_e:\R\to\R,\qquad
\phi_e(x)=cx+y_e,\qquad e\in E,
\]
with no common fixed point, such that
\[
\sigma=\sum_{e\in E} \lambda_e(\phi_e)_*\sigma.
\]
\end{defn}
\begin{defn}
\label{defn:balanced}
    Let $\Psi= (\psi_1, \ldots, \psi_m)$ be an $m$-tuple of non-increasing functions from $(0,\infty)$ to $(0,1)$. Then
    \begin{enumerate}
        \item $\Psi$ is called $\alpha$-balanced for $\alpha \in (0,1)$ if for all $1 \leq i \leq m$ and $a \in (0,\infty)$, we have
        \[
        \psi_i(a) \leq (\psi_1(a) \cdots \psi_m(a))^\alpha.
        \]
        \item $\Psi$ is called $\beta$-proportional with $\beta>1$ if for all $1 \leq i,j \leq m$ and $a,b \in (0,\infty)$, we have
        \[
        \frac{\log(\psi_i(ab)) - \log(\psi_i(a))}{\log(\psi_j(ab)) - \log(\psi_j(a))} \leq \beta.
        \]
        Here we use the notation that $0/0:=1$ and $c/0= \infty$ for all $c \in \R \setminus \{0\}$.
    \end{enumerate}
\end{defn}

\begin{rem}
\label{rem:balanced}
Writing $\psi=\psi_1\cdots\psi_m$ and multiplying the $m$ inequalities in
Definition~\ref{defn:balanced}(1), we obtain $\psi(a)\leq\psi(a)^{m\alpha}$; since
$\psi(a)<1$, this forces $\alpha\leq1/m$. Moreover, taking
logarithms and dividing by $\log\psi(a)<0$ shows that $\alpha$-balance is equivalent to
\[
    w_i(a):=\frac{\log\psi_i(a)}{\log\psi(a)}\geq\alpha,
    \qquad 1\leq i\leq m .
\]
Since $w_1(a)+\cdots+w_m(a)=1$, this says precisely that $(w_1(a),\ldots,w_m(a))$ is a weight
belonging to $[\alpha,1]^m$. It is in this form that $\alpha$-balance will be used; see the
proofs of Proposition~\ref{prop:convergence case} and Lemma~\ref{lem:imp change}.
\end{rem}

\begin{exm}
\label{exm:power}
The model case is that of power functions: for $a_1,\ldots,a_m>0$ with
$a_1+\cdots+a_m=1$, let
\[
    \psi_i(q)=q^{-a_i},\qquad q>1,\ 1\leq i\leq m ,
\]
so that $\psi(q)=q^{-1}$. Then $\psi_i(q)\leq\psi(q)^\alpha$ reads $a_i\geq\alpha$, so that
$\Psi$ is $\alpha$-balanced precisely for $\alpha\leq\min_i a_i$; note that
$\min_ia_i\leq1/m$, in accordance with Remark~\ref{rem:balanced}. Moreover \[
    \frac{\log\psi_i(ab)-\log\psi_i(a)}{\log\psi_j(ab)-\log\psi_j(a)}
    =\frac{a_i}{a_j}
\]
whenever $b\neq1$, so that $\Psi$ is $\beta$-proportional whenever
$\beta\geq\max_ia_i/\min_ia_i$ and $\beta>1$. Thus Theorem~\ref{main thm:counting} applies to every such
$\Psi$, the classical unweighted case being $a_1=\cdots=a_m=1/m$. The
values of $\psi_i$ on $(0,1]$ are immaterial and have been left unspecified:
they influence $\cN_T(\theta,\Psi)$ only through the solutions
of~\eqref{eq:intro1} with $q=1$, of which there are finitely many.
\end{exm}

The main arithmetic result of the paper is the following.
      \begin{thm}
      \label{main thm:counting}
       Let $\nu_1,\ldots,\nu_m$ be self-similar probability measures on $\R$
       sharing a common contraction ratio $c\in(0,1)$, and set $\nu=\nu_1\otimes\cdots\otimes\nu_m$. Let $\Psi= (\psi_1, \ldots, \psi_m)$ be an $m$-tuple of non-increasing functions from $(0,\infty)$ to $(0,1)$. For $\theta \in \R^m$, let $\cN_T(\theta, \Psi )$ denote the number of  integer vectors $(p_1, \ldots, p_m,q)\in\Z^m\times\N$ satisfying
\begin{align}
\label{eq:intro1}
    |p_i+q\theta_i| \le \psi_i(q)
    \quad (1\le i\le m),
    \qquad
    q \le T.
\end{align}
       Then the following holds.
       \begin{enumerate}
           \item If $\Psi$ is $\alpha$-balanced for some $\alpha \in (0,1)$, and
           \[
           \sum_{q\in\N} \psi_1(q) \cdots \psi_m(q) <\infty,
           \]
           then for $\nu$-almost every $\theta \in \R^m$, we have
           \[
        \sup_{T>0}\ \cN_T(\theta, \Psi)<\infty ;
        \]
       \item If $\Psi$ is $\alpha$-balanced and $\beta$-proportional for $\alpha \in (0,1)$ and $\beta>1$, and
       \[
       \sum_{q\in\N} \psi_1(q) \cdots \psi_m(q) =\infty,
       \]
       then for $\nu$-almost every $\theta \in \R^m$, we have
        \begin{align}
        \label{eq:main thm asym}
        \cN_T(\theta, \Psi ) &\mathrel{\sim} 2^m \sum_{q=1}^T \psi_1(q) \cdots \psi_m(q),
        \end{align}
        as $T \rightarrow \infty$.
       \end{enumerate}
    \end{thm}

 Neither $\alpha$-balance nor $\beta$-proportionality has a counterpart in the classical results quoted above; they are artefacts of the dynamical method,
confining the diagonal elements produced by the Dani correspondence to a
compact subcone of the positive Weyl chamber.

   Theorem~\ref{main thm:counting} is proved via homogeneous dynamics: we convert the problem into a lattice point counting problem and use the Borel--Cantelli argument to deduce the count for almost every point. To make it precise, we introduce the following notation.
Let
\[
    G=\SL_{m+1}(\R), \qquad \Gamma=\SL_{m+1}(\Z),
\]
and let $\X=G/\Gamma$, which can be identified with the space of all
unimodular lattices in $\R^{m+1}$ via the map
\[
    g\Gamma \mapsto g\Z^{m+1}.
\]
We denote by $\mu_{\X}$ the unique $G$-invariant probability measure on $\X$.
Throughout the paper, we fix a right-invariant metric on $G$ and let it
induce a metric on $\X$. We define the \emph{injectivity radius} of
$x\in\X$ by
\[
    \inj(x)
    =
    \sup\left\{
        \varepsilon>0:
        g\mapsto gx
        \text{ is injective on the $\varepsilon$-ball centered at the identity in $G$}
    \right\}.
\]

    \begin{defn}
   Throughout this paper, a vector $\bfw= (w_1,\ldots, w_m) \in \R^m$ will be called a \emph{weight} if $w_i >0$ for all $i$ and
    $w_1+ \ldots+ w_m=1$.
\end{defn}

    For a weight $\bfw= (w_1, \ldots, w_m)$, we define
 \[
 g_t(\bfw)= \begin{pmatrix}
        e^{tw_1} \\ & \ddots \\ && e^{tw_m} \\ &&&e^{-t}
    \end{pmatrix},
 \]
 and define for $\theta \in \R^m$, the matrix
 \[
 u(\theta)= \begin{pmatrix}
     \id_m & \theta \\ & 1
 \end{pmatrix}.
 \]
This is the flow used in \cite[(1.5)]{KhalilWeiss}, and the
condition $\bfw\in[\alpha,1]^m$, which will be imposed throughout, is a
quantitative form of the requirement, going back to~\cite{KW08}, that a sequence
of diagonal elements \emph{drift away from the walls} of the Weyl chamber.

 The main dynamical result of the paper is the following.
 \begin{thm}
\label{thm:Double equi}
Suppose that $\nu=\nu_1\otimes\cdots\otimes\nu_m$,
where each $\nu_i$ is a self-similar probability measure with equal
contraction ratio. Suppose further that there exist a weight $\bfw$,
constants $\delta_{\bfw},c_{\bfw}>0$ and an integer $l_0\geq1$ such that for every $f\in C_c^\infty(\X)$, $x\in\X$, and
$t\geq0$,
\begin{align}
\label{eq:e3}
\int_{\R^m}
f(g_t(\bfw)u(\theta)x)\,d\nu(\theta)
=
\mu_{\X}(f)
+
O\!\left(
e^{-\delta_{\bfw}t}
\inj(x)^{-c_{\bfw}}
\|f\|_{C^{l_0}}
\right).
\end{align}

Then, for every $\alpha\in(0,1)$, there exist constants
$\delta_\alpha,c_\alpha>0$ such that, for every pair of weights
$\bfv,\bfv'\in[\alpha,1]^m$, every
$F_1,F_2\in C_c^\infty(\X)$, $x_1,x_2\in\X$, and $t\geq r\geq0$, we have
\begin{align}
\int_{\R^m}
F_1(g_t(\bfv)u(\theta)x_1)
F_2(g_r(\bfv')u(\theta)x_2)\,d\nu(\theta)&=
\mu_{\X}(F_1)\mu_{\X}(F_2) \nonumber \\
&+
O\!\left(
e^{-\delta_\alpha \rho((t,\bfv),(r,\bfv'))} \inj(x_1)^{-c_\alpha}
\|F_1\|_{C^{l_0}}\|F_2\|_{C^1}
\right) \nonumber\\
&+
O\!\left(
e^{-\delta_\alpha r}
\inj(x_2)^{-c_\alpha}
\bigl|\mu_{\X}(F_1)\bigr| \|F_2\|_{C^{l_0}}
\right), \label{eq:dd3}
\end{align}
where
\begin{align}
    \label{eq:def rho}
    \rho((t,\bfv),(r,\bfv'))
=
\min_{1\leq i\leq m}
\left\{
t+t v_i-r-rv_i'
\right\}.
\end{align}
\end{thm}

Theorem~\ref{thm:Double equi} is the main contribution of the paper: it
upgrades effective equidistribution for a \emph{single} weight $\bfw$ to
effective \emph{double} equidistribution, uniformly over all pairs of weights
in the cone $[\alpha,1]^m$. We refer to this passage from a single weight to a
whole cone of weights as a \emph{cone upgrade}. In the present setting the hypothesis~\eqref{eq:e3}
is supplied by~\cite[Thm.~1.2]{benard2026} with $\bfw=(1/m,\ldots,1/m)$, and
Theorem~\ref{main thm:counting} is deduced from the conclusion in
Sections~\ref{sec:con case}--\ref{sec:proof of main}.

The contraction ratios of $\nu_1,\ldots,\nu_m$ are not assumed here to be
equal to one another; the common ratio of
Theorem~\ref{main thm:counting} is needed only in
Section~\ref{sec:proof of main}, in order to verify~\eqref{eq:e3}.

Assuming Theorem~\ref{thm:Double equi}, we briefly indicate how the proof of Theorem~\ref{main thm:counting} follows. By the Dani correspondence, the number of solutions to~\eqref{eq:intro1} can be approximated by counting short vectors in the lattices
\[
g_t(\bfw)u(\theta)\Gamma,
\]
where both the notion of shortness and the diagonal element $g_t(\bfw)$ are determined by $\Psi$; see Sections~\ref{sec:con case} and~\ref{sec:div case} for details.

To estimate the same, we combine \cite[Thm.~1.2]{benard2026} with Theorem~\ref{thm:Double equi} to obtain effective double equidistribution for $\nu$, uniformly over diagonal elements whose weights range over a fixed cone in the positive Weyl chamber. The resulting estimates are obtained by evaluating $g_t(\bfw)u(\theta)\Gamma$ against suitable Siegel transforms.

The $\alpha$-balanced property of $\Psi$ ensures that the corresponding diagonal elements remain in a fixed cone, and hence allows Theorem~\ref{thm:Double equi} to be applied uniformly. This gives the convergence case.

For the divergence case, we require the additional assumption that $\Psi$ is
$\beta$-proportional. It guarantees that the separation of the
diagonal elements $g_t(\bfw)$ and $g_s(\bfv)$ arising from $\Psi$ is
comparable with the separation of the times,
\[
    \rho((t,\bfw),(s,\bfv))\gg t-s
    \qquad (t\geq s),
\]
the implied constant depending only on $m$ and $\beta$. Since $\rho$ is a
minimum over the coordinates, this is not automatic; it is what makes the error
term of Theorem~\ref{thm:Double equi} decay in the separation of the times, and
hence makes the correlation estimate in the Borel--Cantelli argument usable.

We remark that the derivation of Theorem~\ref{main thm:counting} from Theorem~\ref{thm:Double equi} closely parallels the proof in~\cite{benard2026}.

\begin{rem}
Using the techniques developed in this paper, it is possible to extend Theorem~\ref{thm:Double equi} to
multi-equidistribution. More
precisely, one can obtain an effective estimate for integrals of the form
\[
\int_{\R^m}
f_0(\theta)
\prod_{j=1}^n
f_j\bigl(g_{t_j}(\bfw_j)u(\theta)x_j\bigr)
\,d\nu(\theta),
\]
where $\bfw_1,\ldots,\bfw_n\in[\alpha,1]^m$,
$t_1,\ldots,t_n>0$, $x_1,\ldots,x_n\in\X$, $f_0\in C_b^\infty(\R^m)$ and
$f_1,\ldots,f_n\in C_c^\infty(\X)$.

Such an estimate can be obtained by repeated applications of
Propositions~\ref{prop:height function} and~\ref{prop: inductive},
but we do not pursue it here: double
equidistribution already suffices for
Theorem~\ref{main thm:counting}.
\end{rem}

\medskip

\noindent {\bf Acknowledgements.}
The authors would like to thank Seonhee Lim for helpful discussions
during the initial stages of the project. We also thank the MATRIX Institute
and the organisers of the \emph{Ergodic Theory, Diophantine Approximation and
Related Topics 2026} program for their hospitality and for providing a
stimulating environment in which this work was initiated. The first author
would also like to thank Shreyasi Datta, Manuel Hauke-Treuer, and Benjamin
Ward for organising the Great Ocean Road trip, which was a particularly
enjoyable and memorable experience.
 The draft of this paper was proofread and polished with
the assistance of Claude.

\medskip
\section{Notation and preliminaries}
\label{sec:notation}

\subsection{Self-similar measures and their decompositions}
\label{sec:fractals}

Throughout the paper, for each $1\leq i\leq m$, we fix a self-similar
probability measure $\nu_i$ on $\R$ with equal contraction ratio $c_i$ and
support $\mathcal F_i$. We set
\[
    \mathcal F=\prod_{i=1}^m\mathcal F_i,
    \qquad
    \nu=\bigotimes_{i=1}^m\nu_i.
\]

For a diagonal matrix
\[
d=
\begin{pmatrix}
d_1&&&\\
&\ddots&&\\
&&d_m&\\
&&&(d_1\cdots d_m)^{-1}
\end{pmatrix},
\]
with $d_i\geq1$ for all $i$, we define $n_i(d)$ to be the unique integer satisfying
\[
    c_i^{-n_i(d)}
    \leq
    d_i(d_1\cdots d_m)
    \mathrel<
    c_i^{-n_i(d)-1}.
\]
We also define
\begin{align}
    \label{eq: def A}
    A^{(d)}
    =
    \operatorname{diag}
    \bigl(
    c_1^{-n_1(d)},
    \ldots,
    c_m^{-n_m(d)}
    \bigr),
\end{align}
and
\begin{align}
\label{eq:def tilde d}
    \widetilde d
    =
    \det\bigl(A^{(d)}\bigr)^{-1/(m+1)}
    \begin{pmatrix}
    A^{(d)}&\\
    &1
    \end{pmatrix}.
\end{align}

By construction, each diagonal entry of $d\widetilde d^{-1}$ belongs to the
compact interval
\[
    \Bigl[(c_1\cdots c_m)^{\frac{1}{m+1}},\
    \max_{1\leq i\leq m}c_i^{-1}\Bigr]
    \subset
    \bigl[c_1\cdots c_m,(c_1\cdots c_m)^{-1}\bigr].
\]
Indeed, with $\Delta=d_1\cdots d_m$ and $P=\det\bigl(A^{(d)}\bigr)$, the diagonal
entries of $d\widetilde d^{\,-1}$ are $P^{\frac1{m+1}}d_ic_i^{n_i(d)}$ and
$P^{\frac1{m+1}}\Delta^{-1}$, while the definition of $n_i(d)$ gives
$\Delta^{-1}\leq d_ic_i^{n_i(d)}\leq c_i^{-1}\Delta^{-1}$ and, multiplying these over
$1\leq i\leq m$, also
$(c_1\cdots c_m)^{\frac1{m+1}}\leq P^{\frac1{m+1}}\Delta^{-1}\leq1$.

We shall need the following elementary but fundamental
decomposition, which we state for a single self-similar measure on $\R$.

\begin{lem}
\label{lem:self-similar decomposition}
Let $\sigma$ be a self-similar probability measure on $\R$ with equal contraction ratio, say
\[
    \sigma=\sum_{e\in E}\lambda_e\,(\phi_e)_*\sigma,
    \qquad
    \phi_e(x)=cx+y_e ,
\]
where $c\in(0,1)$, $E$ is a finite set, $(\lambda_e)_{e\in E}$ is a probability vector and
$y_e\in\R$. For an integer $N\geq0$ and a word
$\mathbf e=(e_1,\ldots,e_N)\in E^N$ put
\[
    \lambda_{\mathbf e}=\lambda_{e_1}\cdots\lambda_{e_N},
    \qquad
    y_{\mathbf e}=\sum_{k=1}^{N}c^{\,k-1}y_{e_k},
\]
and define the finitely supported probability measure
\[
    \sigma^{(N)}=\sum_{\mathbf e\in E^N}\lambda_{\mathbf e}\,\delta_{y_{\mathbf e}} .
\]
Then, for every bounded measurable function $f:\R\to\R$,
\begin{align}
\label{eq:self-similar decomposition}
    \int_{\R}f(x)\,d\sigma(x)
    =
    \int_{\R}\int_{\R}f\bigl(c^{N}x+y\bigr)\,d\sigma(x)\,d\sigma^{(N)}(y).
\end{align}
\end{lem}

\begin{proof}
Since all the maps $\phi_e$ have the same linear part $c$, an $N$-fold composition has linear
part $c^{N}$ and differs from the other $N$-fold compositions only by a translation. Precisely,
\[
    \phi_{e_1}\circ\cdots\circ\phi_{e_N}(x)=c^{N}x+y_{\mathbf e}
    \qquad\text{for all }\mathbf e\in E^N ,
\]
as one checks by induction on $N$: the case $N=0$ is trivial, and if the formula holds for
$(e_2,\ldots,e_N)$ then
\[
    \phi_{e_1}\bigl(c^{N-1}x+y_{(e_2,\ldots,e_N)}\bigr)
    =c^{N}x+c\,y_{(e_2,\ldots,e_N)}+y_{e_1}
    =c^{N}x+y_{\mathbf e}.
\]
Iterating the self-similarity relation $N$ times gives
\[
    \sigma=\sum_{\mathbf e\in E^N}\lambda_{\mathbf e}\,
    \bigl(\phi_{e_1}\circ\cdots\circ\phi_{e_N}\bigr)_*\sigma
    =\sum_{\mathbf e\in E^N}\lambda_{\mathbf e}\,
    \bigl(c^{N}\,\cdot+\,y_{\mathbf e}\bigr)_*\sigma ,
\]
and integrating $f$ against both sides yields~\eqref{eq:self-similar decomposition}.
\end{proof}

The measure $\sigma^{(N)}$ is thus the distribution of the translation part of a random
cylinder of level $N$: it records \emph{where} the level-$N$ copies of $\sigma$ sit, each such
copy being a rescaling of $\sigma$ by the factor $c^{N}$. Note that $\sigma^{(0)}=\delta_0$,
and that $y_{\mathbf e}+c^{N}\supp(\sigma)\subseteq\supp(\sigma)$ for every
$\mathbf e\in E^{N}$, so that $\supp(\sigma^{(N)})$ lies within distance
$c^{N}\sup_{x\in\supp(\sigma)}|x|$ of $\supp(\sigma)$.

We shall also need the following standard Frostman-type bound.

\begin{lem}
\label{lem:frostman}
Let $\sigma$ be a self-similar probability measure on $\R$ with equal
contraction ratio $c\in(0,1)$, in the sense of Definition~\ref{defn:ss}.
Then there exist $C,\delta>0$ such that
\[
    \sigma([x-y,x+y])\leq C\, y^{\delta}
    \qquad\text{for all } x\in\R,\ y>0 .
\]
In particular $\sigma$ is non-atomic.
\end{lem}

\begin{proof}
This is \cite[Lem.~2.1(ii)]{benard24}, applied to the randomised
system $\sum_{e\in E}\lambda_e\delta_{\phi_e}$: the latter is contractive and
finitely supported, and it is irreducible because on the line irreducibility
means precisely that the maps $\phi_e$ have no common fixed point. For such
systems a short self-contained proof is given in~\cite[Prop.~2.2]{FengLau};
see also \cite[Prop.~4.2]{benard2026} for the corresponding statement near
affine subspaces of $\R^d$.
\end{proof}

For each $i\in\{1,\ldots,m\}$, we define $\nu_i^{(d)}$ to be the probability
measure obtained by applying Lemma~\ref{lem:self-similar decomposition} to
$\sigma=\nu_i$ with $N=n_i(d)$; by~\eqref{eq:self-similar decomposition} it satisfies
\begin{align}
\label{eq:c1}
 \int_{\R}f(x)\,d\nu_i(x)
 =
 \int_{\R}\int_{\R}
 f(c_i^{n_i(d)}x+y)
 \,d\nu_i(x)\,d\nu_i^{(d)}(y).
\end{align}
Identity~\eqref{eq:c1}, together with its coordinatewise
consequence~\eqref{eq:c2} and the resulting decomposition~\eqref{eq: measure decompose}, is
fundamental for everything that follows: it is what converts a translate of the fractal by a
large diagonal element into an average of \emph{bounded} unipotent perturbations of translates
by $\widetilde d$, and it is used in this form in the proofs of
Corollary~\ref{cor:iteration}, Proposition~\ref{prop:height function}
and Proposition~\ref{prop: inductive}.

We define
\[
    \nu^{(d)}
    =
    \bigotimes_{i=1}^m\nu_i^{(d)}.
\]
Applying~\eqref{eq:c1} coordinatewise, we obtain
\begin{align}
\label{eq:c2}
    \int_{\R^m}f(\theta)\,d\nu(\theta)
    =
    \int_{\R^m}\int_{\R^m}
    f((A^{(d)})^{-1}\theta+\phi)
    \,d\nu(\theta)
    \,d\nu^{(d)}(\phi).
\end{align}

Combining~\eqref{eq:c2} with the identity
\[
    \widetilde d\,u(\theta)\widetilde d^{-1}
    =
    u(A^{(d)}\theta),
\]
we obtain
\begin{align}
\label{eq: measure decompose}
    \int_{\R^m}
    f(\widetilde d\,u(\theta)x)
    \,d\nu(\theta)
    =
    \int_{\R^m}\int_{\R^m}
    f(u(\theta)\widetilde d\,u(\phi)x)
    \,d\nu(\theta)
    \,d\nu^{(d)}(\phi).
\end{align}

Finally, fix $M>0$ such that
$\operatorname{supp}(\nu)\subset[-M,M]^m$, and let
\begin{align}
\label{eq:def Omega}
    \Omega
    =
    \bigl\{
    \operatorname{diag}(d_1,\ldots,d_{m+1})\,u(\theta):\
    \theta\in[-M,M]^m,\
    d_i\in[c_1\cdots c_m,(c_1\cdots c_m)^{-1}],\
    d_1\cdots d_{m+1}=1
    \bigr\},
\end{align}
a compact subset of $G$.

\subsection{Height function}
\label{subsec:Height function}

The notation introduced in this subsection will be used only in
Section~\ref{sec:non div}.

For $1\leq l\leq m+1$, let
\[
    V_l=\bigwedge^l\R^{m+1},
    \qquad
    V=\bigoplus_{l=1}^{m+1}V_l.
\]
We equip $V_l$ with the natural action of $G$ induced by
$g\mapsto\bigwedge^l g$, and equip $V$ with the corresponding direct-sum
action. Let $\{\bfe_1,\ldots,\bfe_{m+1}\}$ denote the standard basis of
$\R^{m+1}$. For each index set
\[
    I=\{i_1<\cdots<i_l\}\subset\{1,\ldots,m+1\},
\]
we define
\[
    \bfe_I=\bfe_{i_1}\wedge\cdots\wedge\bfe_{i_l}.
\]
The collection of $\bfe_I$ with $\#I=l$ forms a basis of $V_l$.

For $v\in V$, write
\[
    v=\sum_I v_I\bfe_I,
\]
where the sum is over all index sets $I$, and define the norm
\[
    \|v\|=\max_I |v_I|.
\]
For $g\in G$, we define the corresponding operator norm by
\[
    \|g\|
    =
    \sup\{\|gv\|:v\in V,\ \|v\|=1\}.
\]
For a compact subset $Q\subset G$, we set
\[
    \|Q\|
    =
    \sup_{g\in Q}\max\{\|g\|,\|g^{-1}\|\}.
\]

For a discrete subgroup $\Lambda\leq\R^{m+1}$ of rank $l\geq1$, let
$v_\Lambda\in V_l/\{\pm1\}$ be the wedge product
\[
    v_\Lambda=v_1\wedge\cdots\wedge v_l,
\]
where $v_1,\ldots,v_l$ is any $\Z$-basis of $\Lambda$. This is independent
of the choice of basis. We define
\[
    \|\Lambda\|=\|v_\Lambda\|,
    \qquad
    \|\{0\}\|=1.
\]
Thus, $\|\Lambda\|$ is the covolume of $\Lambda$ in its real span, up to
the choice of the above norm.

For $\Lambda\in\X$, let $P(\Lambda)$ denote the collection of primitive
subgroups of $\Lambda$, namely,
\[
    P(\Lambda)
    =
    \left\{
        L\leq\Lambda:
        L=\Lambda\cap\operatorname{span}_{\R}(L)
    \right\}.
\]
For $1\leq i\leq m$ and $x=g\Gamma\in\X$, define
\[
    \lambda_i(x)
    =
    \sup\left\{
        \|\Lambda_i\|^{-1}:
        \Lambda_i\in P(g\Z^{m+1}),
        \ \operatorname{rank}(\Lambda_i)=i
    \right\}.
\]

Given $\boldsymbol{\hat{\eta}}=(\eta_1,\ldots,\eta_m)\in\R_{>0}^m$ and
$\varepsilon>0$, we define the height function
\begin{align}
    \label{eq:def f e eta}
      f_{\varepsilon,\boldsymbol{\hat{\eta}}}(x)
    =
    \varepsilon^{-1}
    +
    \sum_{i=1}^m\lambda_i(x)^{\eta_i}.
\end{align}

The following lemma relates the height function to the injectivity radius.

\begin{lem}
\label{lem:inj height}
For every $\boldsymbol{\hat{\eta}}\in\R_{>0}^m$ and $\varepsilon>0$, there
exist constants $\kappa_1,\kappa_2>0$ such that, for every $x\in\X$,
\begin{align}
\label{eq:a5}
    \inj(x)^{-\kappa_1}
    \ll
    f_{\varepsilon,\boldsymbol{\hat{\eta}}}(x)
    \ll
    \inj(x)^{-\kappa_2}.
\end{align}
\end{lem}
\begin{proof}
The estimate~\eqref{eq:a5} is well known; we include a proof for completeness.

First, by definition,
\[
    f_{\varepsilon,\boldsymbol{\hat{\eta}}}(x)
    \geq
    \lambda_1(x)^{\eta_1}.
\]
On the other hand, by Minkowski's second theorem, for every
$1\leq i\leq m$,
\[
    \lambda_i(x)\ll \lambda_1(x)^i.
\]
Hence
\[
    f_{\varepsilon,\boldsymbol{\hat{\eta}}}(x)
    \ll
    1+\max_{1\leq i\leq m}\lambda_i(x)^{\eta_i}
    \ll
    \lambda_1(x)^C
\]
for some $C>0$ depending only on $\boldsymbol{\hat{\eta}}$. Thus, it remains
to compare $\lambda_1(x)$ with $\inj(x)$.

To this end, recall that every $x\in\X$ can be written as $x=g\Gamma$, where
\[
    g=kan,
\]
with
\[
    k\in K,\qquad
    a\in A_{2/\sqrt{3}},\qquad
    n\in N_{1/2},
\]
where
\[
    K=\mathrm{SO}(m+1),
\]
\[
    A_{2/\sqrt{3}}
    =
    \left\{
        \operatorname{diag}(a_1,\ldots,a_{m+1}):
        a_i>0,\ a_1\cdots a_{m+1}=1,\
        \frac{a_i}{a_{i+1}}\leq\frac{2}{\sqrt3}
    \right\},
\]
and
\[
    N_{1/2}
    =
    \left\{
        (n_{ij})\in\SL_{m+1}(\R):
        n_{ij}=0\ (i>j),\
        n_{ii}=1,\
        n_{ij}\in[-1/2,1/2]\ (i<j)
    \right\}.
\]
Writing
\[
    g\Gamma=k(ana^{-1})a\Gamma,
\]
we observe that both $k$ and $ana^{-1}$ range over fixed compact subsets of
$\SL_{m+1}(\R)$. Hence, by the compactness of these sets, $\inj(x)$ and
$\lambda_1(x)$ are comparable with $\inj(a\Gamma)$ and $\lambda_1(a\Gamma)$,
respectively.

For
\[
    a=\operatorname{diag}(a_1,\ldots,a_{m+1}),
\]
we have
\[
    \inj(a\Gamma)
    \asymp
    \min_{i\neq j}\frac{a_i}{a_j}
    =
    \frac{\min_i a_i}{\max_j a_j},
\]
while
\[
    \lambda_1(a\Gamma)^{-1}=\min_i a_i.
\]
Indeed, the shortest nonzero vectors in $a\Z^{m+1}$ are the coordinate
vectors. Since
\[
    a_1\cdots a_{m+1}=1,
\]
we have
\[
    1\leq\max_j a_j\leq(\min_i a_i)^{-m}.
\]
Consequently,
\[
    \lambda_1(a\Gamma)^{-(m+1)}
    \ll
    \inj(a\Gamma)
    \ll
    \lambda_1(a\Gamma)^{-1}.
\]
Combining these estimates with the preceding bounds on
$f_{\varepsilon,\boldsymbol{\hat{\eta}}}$ proves the lemma for suitable
constants $\kappa_1,\kappa_2>0$.
\end{proof}

\subsection{Norms}
\label{subsec:norms}
For a function $f:\R^m\to\R$, $t>0$ and a weight $\bfw$, we define
\[
    \|f\|_{\bfw,t}
    =
    \sup_{\substack{\phi\in\R^m\\ \theta\in\supp(\nu)}}
    \left|
        f\bigl((A^{(g_t(\bfw))})^{-1}\theta+\phi\bigr)-f(\phi)
    \right|,
\]
where $A^{(\cdot)}$ is defined as in~\eqref{eq: def A}.
Here
$\bigl(A^{(g_t(\bfw))}\bigr)^{-1}
=\operatorname{diag}\bigl(c_1^{n_1(g_t(\bfw))},\ldots,c_m^{n_m(g_t(\bfw))}\bigr)$
is a contraction whose entries tend to $0$ as $t\to\infty$, so that
$\|f\|_{\bfw,t}$ measures the oscillation of $f$ at a scale shrinking with $t$;
in particular it is small for large $t$ whenever $f$ is uniformly continuous.

For $k\geq0$, let $C_b^k(\R^m)$ denote the space of bounded
$k$-times continuously differentiable functions on $\R^m$. We equip this
space with the norm
\[
    \|f\|_{C^k}
    =
    \max_{\substack{0\leq j\leq k\\ 1\leq i_1,\ldots,i_j\leq m}}
    \left\|
        \frac{\partial^j f}
        {\partial x_{i_1}\cdots\partial x_{i_j}}
    \right\|_{C^0},
    \qquad
    \|f\|_{C^0}
    =
    \sup_{x\in\R^m}|f(x)|.
\]

We next define the norms on $\X$. For $k \in \N \cup\{\infty\}$, let $C_c^k(\X)$ denote the space of compactly supported $k$-times continuously differentiable functions on $\X$. For $Y\in\Lie(G)$, let $D_Y$ denote the
first-order differential operator on $C_c^\infty(\X)$ given by
\[
    D_Yf(y)
    :=
    \left.\frac{d}{dt}f(\exp(tY)y)\right|_{t=0}.
\]
Fix an ordered basis $\{Y_1,\ldots,Y_r\}$ of $\Lie(G)$. For a monomial
\[
    Z=Y_1^{\ell_1}\cdots Y_r^{\ell_r},
\]
define
\[
    D_Z
    :=
    D_{Y_1}^{\ell_1}\cdots D_{Y_r}^{\ell_r},
    \qquad
    \deg(Z):=\ell_1+\cdots+\ell_r.
\]
We denote the uniform norm on $C_c(\X)$ by
\[
    \|f\|_{C^0}
    :=
    \sup_{x\in\X}|f(x)|,
\]
and, for $k\geq1$, define the $C^k$ norm on $C_c^k(\X)$ by
\[
    \|f\|_{C^k}
    :=
    \sum_{\deg(Z)\leq k}\|D_Zf\|_{C^0}.
\]
In the hypotheses of Theorem~\ref{thm:Double equi} and of
Propositions~\ref{prop: inductive} and~\ref{prop:change of weights} the
integer $l_0\geq1$ is quantified together with $\delta_{\bfw}$ and
$c_{\bfw}$, exactly as in the hypotheses of
Propositions~\ref{prop:convergence case} and~\ref{prop:div case}. For the
measures $\nu$ considered in Theorem~\ref{main thm:counting} one may take
$l_0=\lceil\tfrac12\dim SO(m+1)\rceil$, by \cite[Thm.~1.2]{benard2026}.

Finally, for $k\geq1$ and $f\in C_c^k(\X)$, define the $L^2$-Sobolev
norm by
\begin{align*}
    \mathcal S_{2,k}(f)
    &:=
    \left(
        \sum_{\deg(Z)\leq k}
        \int_{\X}|D_Zf(y)|^2\,d\mu_{\X}(y)
    \right)^{1/2}.
\end{align*}

We now recall the exponential mixing estimate that will be needed in the proof.

\begin{thm}
\label{thm:mixing}
There exist $\delta_2>0$ and $l'\in\mathbb N$ such that, for every $g\in G$ and all
$f_1,f_2\in C_c^\infty(\X)$,
\begin{align}
\label{eq:mixing}
    \int_\X
    f_1(x)
    f_2(gx)
    \,d\mu_\X(x)
    =
    \mu_\X(f_1)\mu_\X(f_2)+
    O\left(
        \|g\|^{-\delta_2}
        \mathcal{S}_{2,l'}(f_1)
        \mathcal{S}_{2,l'}(f_2)
    \right),
\end{align}
where $\|\cdot\|$ is as in Section~\ref{subsec:Height function}.
\end{thm}
\begin{proof}
This comes from the standard decay of matrix coefficients for
$K$-finite vectors in unitary representations without almost invariant vectors,
which goes back to works of Harish-Chandra, Borel--Wallach, Cowling, Howe and
others; we
refer to~\cite{HoweTan} for a direct self-contained proof in the case
$G=\SL_{m+1}(\R)$. It is applicable to the representation of $G$ on
$L^2_0(\X)$, since $\Gamma$ is a lattice in the simple group $G$.
Katok--Spatzier~\cite{KatokSpatzier94} extended these estimates to smooth
vectors; see also~\cite[Cor.~2.4.4]{KM96} for an explicit estimate.
Any two norms on $G$ of the kind considered here differ by
bounded powers of one another, which only affects the value of $\delta_2$.

\end{proof}

The $L^2$-Sobolev norm is essential here: the test functions to
which~\eqref{eq:mixing} is applied in Section~\ref{sec:div case} are smoothed
Siegel transforms, whose sup-norms are not controlled, while their $L^2$-norms
are, by Proposition~\ref{prop:roger} below.

\subsection{Siegel transforms and moment formulae}

For a bounded measurable function $f:\R^{m+1}\to\R$ with compact support,
we define its \emph{Siegel transform} $\widehat{f}:\X\to\R$ by
\[
    \widehat{f}(g\Gamma)
    =
    \sum_{x\in g\Z^{m+1}\setminus\{0\}}f(x).
\]

We will use the following Siegel mean value theorem~\cite{sie} and
Rogers' second moment formula~\cite{rog}.

\begin{prop}\label{prop:roger}
Let $f\in L^1(\R^{m+1})\cap L^2(\R^{m+1})$ be non-negative. Then
\[
    \int_{\X}\widehat{f}(\Lambda)\,d\mu_{\X}(\Lambda)
    =
    \int_{\R^{m+1}}f(x)\,dx.
\]
Moreover, since $m\geq2$,
\begin{align*}
    \int_{\X}\widehat{f}(\Lambda)^2\,d\mu_{\X}(\Lambda)
    &=
    \left(
        \int_{\R^{m+1}}f(x)\,dx
    \right)^2
    +
    \sum_{\substack{k\in\N,\ q\in\Z\setminus\{0\}\\
                    \gcd(k,q)=1}}
    \int_{\R^{m+1}}f(kx)f(qx)\,dx \\
    &=
    \left(
        \int_{\R^{m+1}}f(x)\,dx
    \right)^2
    +
    O_m\left(
        \|f\|_{\infty}
        \int_{\R^{m+1}}f(x)\,dx
    \right).
\end{align*}
\end{prop}

\subsection{Diagonal elements associated with $\Psi$}
\label{subsec:dani}

The following elementary consequence of $\alpha$-balance, used in
Sections~\ref{sec:con case} and~\ref{sec:div case}, identifies the diagonal
matrices produced by the Dani correspondence as elements of the flow
$g_s(\bfw)$ with $\bfw$ in the cone $[\alpha/2,1]^m$.

\begin{lem}
\label{lem:dani}
Let $\alpha\in(0,1)$, let $Q>1$ and let $z_1,\ldots,z_m\in(0,1)$ satisfy
\[
    z_i\leq(z_1\cdots z_m)^{\alpha}
    \quad(1\leq i\leq m),
    \qquad
    (z_1\cdots z_m)^{2-(2m+1)\alpha}\,Q^{\,2-m\alpha}\geq1 .
\]
Put $z=z_1\cdots z_m$, $r=(z Q)^{1/(m+1)}$ and $s=\log(Q/r)$. Then
\[
    \operatorname{diag}
    \bigl(z_1^{-1}r,\ldots,z_m^{-1}r,\,Q^{-1}r\bigr)
    =g_s(\bfw)
    \qquad\text{for some } \bfw\in[\alpha/2,1]^m .
\]
\end{lem}

\begin{proof}
The matrix has determinant $z^{-1}r^{m+1}Q^{-1}=1$, and its last diagonal
entry is $Q^{-1}r=e^{-s}$, so it equals $g_s(\bfw)$ with
$w_i=\log(z_i^{-1}r)/\log(Q/r)$. Since
$\sum_i\log(z_i^{-1}r)=\log(r^m/z)=\log(Q/r)$ by the definition of $r$,
we have $\sum_iw_i=1$, so $\bfw$ is a weight. Moreover, raising the second
hypothesis to the power $1/(2(m+1))$ and using the first, \[
    z_i
    \leq
    z^{\alpha}
    \bigl(z^{2-(2m+1)\alpha}Q^{2-m\alpha}\bigr)^{\frac{1}{2(m+1)}}
    =
    (z Q)^{\frac{2+\alpha}{2(m+1)}}Q^{-\frac\alpha2}
    =
    r\,(Q/r)^{-\frac\alpha2} ,
\]
so that $z_i^{-1}r\geq(Q/r)^{\alpha/2}$ and hence $w_i\geq\alpha/2$. As the
$w_i$ are positive and sum to $1$, also $w_i\leq1$.
\end{proof}

\medskip
\section{Non-divergence estimates}
\label{sec:non div}

The aim of this section is a quantitative non-divergence estimate
for the translates $g_t(\bfw)u(\phi)x$, uniform over $\bfw$ in the cone
$[\alpha,1]^m$ and over the base point $x$. Its qualitative ancestor is
\cite[Prop.~3.6]{KhalilWeiss}, which asserts, for a compactly supported
$(C,\alpha)$-decaying and Federer (doubling) measure and uniformly over base points,
that there is no escape of mass along any sequence of diagonal elements
\emph{drifting away from the walls} of the Weyl chamber, in the terminology
of~\cite{KW08}; that statement rests in turn on
\cite[Prop.~3.5]{KhalilWeiss} and on the non-divergence estimates
of~\cite{KLW}, which in turn go back to \cite{KM98}. Writing $g_t(\bfw)=\exp(X)$ one has
$\min_{i\leq m}X_i=t\min_iw_i\geq\alpha t$ for $\bfw\in[\alpha,1]^m$, so that
the standing hypothesis $\bfw\in[\alpha,1]^m$ of this paper is precisely a
quantitative form of drifting away from the walls. What follows may thus be
read as an effective version of \cite[Prop.~3.6]{KhalilWeiss} for the class of
measures at hand.

For a weight $\bfw=(w_1,\ldots,w_m)$, $t>0$, $\varsigma>0$, and
$x\in\X$, define
\begin{align}
    J(\bfw,t,\varsigma,x)
    =
    \left\{
        \phi\in\R^m:
        \inj\!\left(
            \widetilde{g_t(\bfw)}u(\phi)x
        \right)>\varsigma
    \right\},
\end{align}
where $\widetilde{g_t(\bfw)}$ is defined as in~\eqref{eq:def tilde d}.

The main goal of this section is to prove the following non-divergence
estimate.

\begin{prop}\label{prop:height function}
For every $0<\alpha<1$, there exist positive constants
$\kappa_1,\kappa_2,\kappa_3$, depending only on $\alpha$ and $\nu$, such that for
every $\bfw\in[\alpha,1]^m$, $t>0$, $\varsigma>0$, and $x\in\X$,
\begin{align}
    \nu^{(g_t(\bfw))}
    \!\left(
        J(\bfw,t,\varsigma,x)^c
    \right)
    \ll
    \varsigma^{\kappa_1}
    \left(
        1+
        \inj(x)^{-\kappa_2}e^{-\kappa_3t}
    \right).
\end{align}
\end{prop}

The following results will be used in the proof of
Proposition~\ref{prop:height function}.

\begin{lem}
\label{lem:contraction}
Let $0<\alpha<1$.
There exist $\boldsymbol{\hat{\eta}}=(\eta_1,\ldots,\eta_m)\in\R_{>0}^m$,
depending only on $\nu$, and constants $\kappa_0>0$ and
$C>0$, depending only on $\alpha$ and $\nu$, such
that the following holds. For every sufficiently large $t$, there exist
$\varepsilon=\varepsilon(t,\alpha,\nu)>0$ and
$b=b(t,\alpha,\nu)>0$ such that the function
$f_{\varepsilon,\boldsymbol{\hat{\eta}}}$ defined in
\eqref{eq:def f e eta} satisfies
\begin{align}
\label{eq:height fun contraction}
    \int_{\R^m}
    f_{\varepsilon,\boldsymbol{\hat{\eta}}}
    \bigl(g_t(\bfw)u(\theta)x\bigr)
    \,d\nu(\theta)
    \leq
    Ce^{-\kappa_0 t}
    f_{\varepsilon,\boldsymbol{\hat{\eta}}}(x)
    +
    b
\end{align}
for every $x\in\X$ and every weight
$\bfw=(w_1,\ldots,w_m)\in[\alpha,1]^m$.
\end{lem}
\begin{proof}
We follow the argument of \cite[Prop.~5.3]{AG24singular}. We recall the
argument and indicate the modifications needed in the present setting.
We use throughout the notation introduced in Section~\ref{subsec:Height function}.

\medskip
\noindent\textbf{Step 0: Notation.}
For $1\leq l\leq m+1$, let $V_l^+$ be the subspace of $V_l$ spanned by
the vectors $\bfe_I$ for which
\[
    \#\bigl(I\cap\{1,\ldots,m\}\bigr)=\min\{l,m\},
\]
and let $V_l^-$ be the subspace spanned by the remaining basis vectors.
We denote by $\pi_{l+}:V_l\to V_l^+$ and
$\pi_{l-}:V_l\to V_l^-$ the corresponding projections.

For $1\leq l\leq m$, define the $l$-th critical exponent
$\zeta_l(\nu)$ of $\nu$ to be the supremum of all $\gamma\geq0$ for which
there exists $C_{\gamma,l}>0$ such that
\begin{align}
\label{eq:def c gamma l}
    \int_{\R^m}
    \frac{1}{\|\pi_{l+}(u(\theta)v)\|^\gamma}
    \,d\nu(\theta)
    \leq C_{\gamma,l}
\end{align}
for every decomposable vector
$v=v_1\wedge\cdots\wedge v_l\in V_l$ with $\|v\|=1$.

\medskip
\noindent\textbf{Step 1: Positivity of the critical exponents.}
The argument is analogous to the proof of \cite[Lem.~4.4]{AG24singular}.
The only change required is to replace \cite[Prop.~2.2]{AG24singular} by
Lemma~\ref{lem:frostman}.
Applying this estimate to the factors of $\nu$ gives
\[
    \zeta_l(\nu)>0
\]
for every $1\leq l\leq m$.

In \cite{AG24singular} the exponents are expressed in terms of the dimension of
$\supp(\nu)$, the measures $\nu_i$ being assumed there to be Hausdorff measures
on self-similar fractals satisfying the open set condition. That assumption is
not needed for the construction of the height function, and only the
identification of the Frostman exponent is lost: writing $\delta_i>0$ for an
exponent as in Lemma~\ref{lem:frostman} for $\nu_i$, the proof
of \cite[Lem.~4.4]{AG24singular} --- which uses nothing beyond compact support,
Fubini's theorem and that bound --- yields the explicit estimate
\[
    \zeta_l(\nu)
    \ \geq\
    \min\Bigl\{\textstyle\sum_{i\in I}\delta_i\ :\
    I\subseteq\{1,\ldots,m\},\ \#I=m+1-l\Bigr\}
    \qquad(1\leq l\leq m).
\]
Since $\boldsymbol{\hat\eta}$, $\varepsilon$, $b$ and $\kappa_0$ below, and hence the
constants $\kappa_1,\kappa_2,\kappa_3$ of
Proposition~\ref{prop:height function}, all descend from the $\zeta_l(\nu)$,
the whole chain of constants is effective in terms of the $\delta_i$.

\medskip
\noindent\textbf{Step 2: Choice of the exponents.}
Choose $\boldsymbol{\hat{\eta}}
    =
    (\eta_1,\ldots,\eta_m)\in\R_{>0}^m$ such that
\[
    0<\eta_i<\zeta_i(\nu), \quad \text{ and } \quad \frac{1}{\eta_{i-j}}
    +
    \frac{1}{\eta_{i+j}}
    <
    \frac{2}{\eta_i}
\]
whenever $1\leq i\leq m$ and $1\leq j\leq\min\{i,m+1-i\}$,
where we use the convention
\[
    \frac{1}{\eta_0}
    =
    \frac{1}{\eta_{m+1}}
    =0.
\]
Such a choice is possible. Indeed, the second condition states that
the finite sequence $\bigl(1/\eta_i\bigr)_{i=0}^{m+1}$ is strictly
concave, and it is satisfied by
\[
    \frac{1}{\eta_i}
    =
    A\,i(m+1-i), \qquad 0\leq i\leq m+1,
\]
for any $A>0$, since in this case
\[
    \frac{1}{\eta_{i-j}}
    +
    \frac{1}{\eta_{i+j}}
    -
    \frac{2}{\eta_i}
    =
    -2Aj^2
    <
    0 .
\]
Moreover $\eta_i\to 0$ as $A\to\infty$, so the first condition
$\eta_i<\zeta_i(\nu)$ also holds once $A$ is large enough, because all
the critical exponents $\zeta_i(\nu)$ are positive by Step~1.

Set
\[
    \kappa_0
    =
    \alpha\min_{1\leq l\leq m}\{l\eta_l\}, \quad \text{ and } \quad  C'
    =
    \max_{1\leq l\leq m} C_{\eta_l,l},
\]
where $C_{\eta_l,l}$ is as in \eqref{eq:def c gamma l}.

The argument in the proof of \cite[Prop.~5.1]{AG24singular} then gives,
for every $1\leq l\leq m$, $t>0$, every decomposable vector
$v=v_1\wedge\cdots\wedge v_l\in V_l\setminus\{0\}$, and
$\bfw\in[\alpha,1]^m$,
\begin{align*}
    \int_{\R^m}
    \|g_t(\bfw)u(\theta)v\|^{-\eta_l}
    \,d\nu(\theta)
    \leq
    C'
    e^{-\kappa_0 t}
    \|v\|^{-\eta_l}.
\end{align*}

\medskip
\noindent\textbf{Step 3: Contraction of the height function.}
Let $D>1$ be a constant such that, for every $\Lambda\in\X$ and
$\Lambda_1,\Lambda_2\in P(\Lambda)$,
\begin{align*}
    \|\Lambda_1\cap\Lambda_2\|
    \|\Lambda_1+\Lambda_2\|
    \leq
    D\|\Lambda_1\|\|\Lambda_2\|.
\end{align*}
The existence of such a constant follows from
\cite[Lem.~5.6]{EMM98}.

For fixed $\alpha\in(0,1)$ and $t>1$ with $C'e^{-\kappa_0 t}<1$, define
\[
    \xi'(\alpha,t)
    =
    \left\|
        \left\{
            g_t(\bfw)u(\theta):
            \theta\in\supp(\nu),\
            \bfw\in[\alpha,1]^m
        \right\}
    \right\|
\]
and
\[
    \xi(\alpha,t)
    =
    \max_{1\leq i\leq m}
    \bigl(D\xi'(\alpha,t)\bigr)^{\eta_i}.
\]
Furthermore, set
\[
    \an
    =
    \min_{\substack{
        1\leq i\leq m\\
        1\leq j\leq\min\{i,m+1-i\}
    }}
    \left\{
        1-
        \frac{\eta_i}{2}
        \left(
            \frac{1}{\eta_{i-j}}
            +
            \frac{1}{\eta_{i+j}}
        \right)
    \right\},
\]
where we use the convention
\[
    \frac1{\eta_0}
    =
    \frac1{\eta_{m+1}}
    =
    0.
\]
By the choice of $\boldsymbol{\hat{\eta}}$ in Step~2, we have
$\an>0$.

Now choose
\[
    \varepsilon
    =
   \left( \frac{C'e^{-\kappa_0 t}}
         {m\,\xi(\alpha,t)} \right)^{1/\an}
\]
and
\[
    b
    =
    \varepsilon^{-1}
    \left(
        1-C'e^{-\kappa_0 t}
    \right).
\]
Following the argument in
\cite[Prop.~5.2 and Prop.~5.3]{AG24singular}, with the above choices,
we obtain
\[
    \int_{\R^m}
    f_{\varepsilon,\boldsymbol{\hat{\eta}}}
    \bigl(g_t(\bfw)u(\theta)x\bigr)
    \,d\nu(\theta)
    \leq
    2C'e^{-\kappa_0 t}
    f_{\varepsilon,\boldsymbol{\hat{\eta}}}(x)
    +b
\]
for every $x\in\X$ and every
$\bfw\in[\alpha,1]^m$.

Thus, upon setting $ C
    =
    2C'$, we obtain the desired contraction estimate, as required.
\end{proof}

\begin{cor}
\label{cor:iteration}
Let $\boldsymbol{\hat{\eta}}$ be as in Lemma~\ref{lem:contraction}. For every
$0<\alpha<1$, there exist $\varepsilon=\varepsilon(\alpha,\nu)>0$ and
$\kappa_3=\kappa_3(\alpha,\nu)>0$ such that, for every $t>0$,
$\bfw\in[\alpha,1]^m$, and $x\in\X$,
\begin{align}
\label{eq:a2}
    \int_{\R^m}
    f_{\varepsilon,\boldsymbol{\hat{\eta}}}
    \bigl(g_t(\bfw)u(\theta)x\bigr)
    \,d\nu(\theta)
    \ll
    e^{-\kappa_3t}
    f_{\varepsilon,\boldsymbol{\hat{\eta}}}(x)
    +1,
\end{align}
where the implied constant is independent of $x$, $t$, and
$\bfw$.
\end{cor}

\begin{proof}
Fix $\alpha\in(0,1)$, and let
$\boldsymbol{\hat{\eta}}$, $C$, and $\kappa_0$ be as
in Lemma~\ref{lem:contraction}. By the log-Lipschitz property of the height
function, there exists $\gamma_0>1$ such that
\[
    \gamma_0^{-1}f_{\varepsilon,\boldsymbol{\hat{\eta}}}(x)
    \leq
    f_{\varepsilon,\boldsymbol{\hat{\eta}}}(gx)
    \leq
    \gamma_0 f_{\varepsilon,\boldsymbol{\hat{\eta}}}(x)
\]
for every $x\in\X$ and every $g$ in the compact set $\Omega$ of~\eqref{eq:def Omega}.

Choose $t_0>1$ sufficiently large so that
\[
    \varrho
    :=
    \gamma_0^2Ce^{-\kappa_0 t_0}
    <1.
\]
Fix $\varepsilon=\varepsilon(t_0,\alpha,\nu)$ and
$b=b(t_0,\alpha,\nu)$ as in Lemma~\ref{lem:contraction}, and abbreviate

\[
    I_n
    =
    \int_{\R^m}
    f_{\varepsilon,\boldsymbol{\hat{\eta}}}
    \bigl(g_{nt_0}(\bfw)u(\theta)x\bigr)\,d\nu(\theta),
    \qquad n\geq0 .
\]
We claim that, for every $n\in\N$ and $x\in\X$,
\begin{align}
\label{eq:cor iteration}
    I_n
    \ll
    \varrho^{\,n} f_{\varepsilon,\boldsymbol{\hat{\eta}}}(x)+1.
\end{align}

Indeed, set $d=g_{(n-1)t_0}(\bfw)$, so that $g_{nt_0}(\bfw)=g_{t_0}(\bfw)d$, and
put $h=d\widetilde d^{\,-1}$. As computed in Section~\ref{sec:fractals}, $h$ is
diagonal with entries in $[c_1\cdots c_m,(c_1\cdots c_m)^{-1}]$; in particular
$h$ commutes with $g_{t_0}(\bfw)$ and $h\,u(\theta)\in \Omega$ for every
$\theta\in\operatorname{supp}(\nu)$. Applying~\eqref{eq: measure decompose} to
the function $y\mapsto f_{\varepsilon,\boldsymbol{\hat{\eta}}}(g_{t_0}(\bfw)h\,y)$
and using $h\widetilde d=d$, we obtain
\[
    I_n
    =
    \int_{\R^m}\int_{\R^m}
    f_{\varepsilon,\boldsymbol{\hat{\eta}}}
    \bigl(
    g_{t_0}(\bfw)h\,u(\theta)\,\widetilde d\,u(\phi)x
    \bigr)
    \,d\nu(\theta)\,d\nu^{(d)}(\phi).
\]
Deleting $h$ by the log-Lipschitz property, and then applying
Lemma~\ref{lem:contraction} to the inner integral with base point
$\widetilde d\,u(\phi)x$, yields
\begin{align*}
    I_n
    &\leq
    \gamma_0
    \int_{\R^m}\int_{\R^m}
    f_{\varepsilon,\boldsymbol{\hat{\eta}}}
    \bigl(g_{t_0}(\bfw)u(\theta)\,\widetilde d\,u(\phi)x\bigr)
    \,d\nu(\theta)\,d\nu^{(d)}(\phi)
    \\
    &\leq
    \gamma_0Ce^{-\kappa_0 t_0}
    \int_{\R^m}
    f_{\varepsilon,\boldsymbol{\hat{\eta}}}
    \bigl(\widetilde d\,u(\phi)x\bigr)\,d\nu^{(d)}(\phi)
    +\gamma_0 b .
\end{align*}
Applying~\eqref{eq: measure decompose} once more, this time to
$y\mapsto f_{\varepsilon,\boldsymbol{\hat{\eta}}}(h\,y)$, and using the
log-Lipschitz property again, we obtain
\[
    I_{n-1}
    =
    \int_{\R^m}\int_{\R^m}
    f_{\varepsilon,\boldsymbol{\hat{\eta}}}
    \bigl(h\,u(\theta)\,\widetilde d\,u(\phi)x\bigr)
    \,d\nu(\theta)\,d\nu^{(d)}(\phi)
    \geq
    \gamma_0^{-1}
    \int_{\R^m}
    f_{\varepsilon,\boldsymbol{\hat{\eta}}}
    \bigl(\widetilde d\,u(\phi)x\bigr)\,d\nu^{(d)}(\phi),
\]
and therefore
\[
    I_n\leq\gamma_0^2Ce^{-\kappa_0 t_0}I_{n-1}+\gamma_0 b=\varrho I_{n-1}+\gamma_0 b .
\]
Since $u(\theta)\in \Omega$ for $\theta\in\operatorname{supp}(\nu)$, we have
$I_0\leq\gamma_0f_{\varepsilon,\boldsymbol{\hat{\eta}}}(x)$, and iterating the
last inequality proves~\eqref{eq:cor iteration}.

Since
\[
    \varrho^n
    =
    e^{-n|\log\varrho|}
    =
    e^{-\kappa_3nt_0},
    \qquad
    \kappa_3=\frac{-\log\varrho}{t_0}>0,
\]
we obtain the desired estimate for all $t$ that are integer multiples of
$t_0$.

Finally, write
\[
    t=nt_0+s,
    \qquad 0\leq s<t_0.
\]
The set of elements $g_s(\bfw)$, with
$0\leq s\leq t_0$ and $\bfw\in[\alpha,1]^m$, is compact. Hence,
by the log-Lipschitz property of
$f_{\varepsilon,\boldsymbol{\hat{\eta}}}$, the estimate at time $nt_0$
extends uniformly to time $t$. This proves \eqref{eq:a2}.
\end{proof}

\begin{proof}[Proof of Proposition~\ref{prop:height function}]
Fix $\alpha\in(0,1)$, and let
$\boldsymbol{\hat{\eta}}$ be as in Lemma~\ref{lem:contraction}. We take $\varepsilon$ and $\kappa_3$ to be the constants of
Corollary~\ref{cor:iteration}, and $\kappa_1,\kappa_2$ those of
Lemma~\ref{lem:inj height}; these are the constants of the statement.

By Lemma~\ref{lem:inj height},
\[
    \inj(y)\leq\varsigma
    \quad\Longrightarrow\quad
    f_{\varepsilon,\boldsymbol{\hat{\eta}}}(y)
    \gg \varsigma^{-\kappa_1}.
\]
Hence, by the definition of $J(\bfw,t,\varsigma,x)$ and Markov's
inequality,
\begin{align*}
\nu^{(g_t(\bfw))}
\bigl(J(\bfw,t,\varsigma,x)^c\bigr)
&\leq
\nu^{(g_t(\bfw))}
\left(
\left\{
\phi:
f_{\varepsilon,\boldsymbol{\hat{\eta}}}
\bigl(
\widetilde{g_t(\bfw)}u(\phi)x
\bigr)
\gg \varsigma^{-\kappa_1}
\right\}
\right)
\\
&\ll
\varsigma^{\kappa_1}
\int_{\R^m}
f_{\varepsilon,\boldsymbol{\hat{\eta}}}
\bigl(
\widetilde{g_t(\bfw)}u(\phi)x
\bigr)
\,d\nu^{(g_t(\bfw))}(\phi)
\\
&\ll
\varsigma^{\kappa_1}
\int_{\R^m}
f_{\varepsilon,\boldsymbol{\hat{\eta}}}
\bigl(
g_t(\bfw)u(\theta)x
\bigr)
\,d\nu(\theta),
\end{align*}
 where the last inequality is obtained exactly as in the proof of
Corollary~\ref{cor:iteration}: applying~\eqref{eq: measure decompose} to
$y\mapsto f_{\varepsilon,\boldsymbol{\hat{\eta}}}
\bigl(g_t(\bfw)\widetilde{g_t(\bfw)}^{-1}y\bigr)$ and using that
$g_t(\bfw)\widetilde{g_t(\bfw)}^{-1}u(\theta)\in \Omega$ for
$\theta\in\operatorname{supp}(\nu)$, the log-Lipschitz property changes the
integral by at most the factor $\gamma_0$. In particular the implied constant
depends only on $\alpha$ and $\nu$, and not on $\bfw\in[\alpha,1]^m$, $x\in\X$
or $t\geq0$.

Applying Corollary~\ref{cor:iteration}, we obtain
\[
    \nu^{(g_t(\bfw))}
    \bigl(J(\bfw,t,\varsigma,x)^c\bigr)
    \ll
    \varsigma^{\kappa_1}
    \left(
        e^{-\kappa_3t}
        f_{\varepsilon,\boldsymbol{\hat{\eta}}}(x)
        +1
    \right).
\]
Finally, Lemma~\ref{lem:inj height} gives
\[
    f_{\varepsilon,\boldsymbol{\hat{\eta}}}(x)
    \ll
    \inj(x)^{-\kappa_2}.
\]
Therefore,
\[
    \nu^{(g_t(\bfw))}
    \bigl(J(\bfw,t,\varsigma,x)^c\bigr)
    \ll
    \varsigma^{\kappa_1}
    \left(
        1+
        \inj(x)^{-\kappa_2}e^{-\kappa_3t}
    \right),
\]
which proves the proposition.
\end{proof}

\medskip

\section{Cone upgrade and double equidistribution}
\label{sec:Double equi}

To prove Theorem~\ref{thm:Double equi}, we will need the following results.

\begin{prop}
\label{prop: inductive}
Suppose that there exist a weight $\bfw$, constants
$\delta_{\bfw},c_{\bfw}>0$ and an integer $l_0\geq1$ such
that~\eqref{eq:e3} holds for every $f\in C_c^\infty(\X)$, $x\in\X$, and
$t\geq0$.

Then, for every $\alpha\in(0,1)$, there exist constants
$\delta_\alpha,c_\alpha>0$ such that, for every
$\bfv\in[\alpha,1]^m$, $F_0\in C_b^\infty(\R^m)$,
$F_1\in C_c^\infty(\X)$, $x\in\X$, and $s,\ell\geq0$, we have
\begin{align}
\int_{\R^m}
F_0(\theta)
F_1\bigl(g_s(\bfv)g_\ell(\bfw)u(\theta)x\bigr)
\,d\nu(\theta)
&=
\nu(F_0)
\mu_{\X}(F_1)
+
O\!\left( \|F_0\|_{\bfv,s}\|F_1\|_{C^0} \right)
\nonumber \\
&\qquad\quad+
O\!\left(
e^{-\delta_\alpha \ell}
\inj(x)^{-c_\alpha}
\|F_0\|_{C^0} \|F_1\|_{C^{l_0}}
\right). \label{eq:dd1}
\end{align}
\end{prop}
\begin{proof}
For $s\geq0$, define
\[
    F_{1,s}(y)
    :=
    F_1\bigl(
        g_s(\bfv)
        \widetilde{g_s(\bfv)}^{-1}y
    \bigr),
\]
where $\widetilde{g_s(\bfv)}$ is defined as in~\eqref{eq:def tilde d}.
Since $g_s(\bfv)\widetilde{g_s(\bfv)}^{-1}$ belongs to a fixed compact subset of $\SL_{m+1}(\R)$, independently of
$s$ and $\bfv\in[\alpha,1]^m$, we have
\begin{align}
\label{eq:e1}
    \|F_{1,s}\|_{C^{l_0}}
    \ll
    \|F_1\|_{C^{l_0}},
\end{align}
where the implied constant depends only on $\alpha$. Moreover, by the
$\SL_{m+1}(\R)$-invariance of $\mu_{\X}$,
\begin{align}
\label{eq:e2}
    \mu_{\X}(F_{1,s})
    =
    \mu_{\X}(F_1).
\end{align}

Applying~\eqref{eq:c2} to the product of $F_0$ and $F_{1,s}$, and
conjugating as in the derivation of~\eqref{eq: measure decompose}, we
obtain
\begin{align*}
&\int_{\R^m}
F_0(\theta)
F_1\bigl(
g_s(\bfv)g_\ell(\bfw)u(\theta)x
\bigr)
\,d\nu(\theta)
\\
&=
\int_{\R^m}
F_0(\theta)
F_{1,s}\bigl(
\widetilde{g_s(\bfv)}
g_\ell(\bfw)u(\theta)x
\bigr)
\,d\nu(\theta)
\\
&=
\int_{\R^m}\int_{\R^m}
F_0\bigl(
\phi+(A^{(g_s(\bfv))})^{-1}\theta
\bigr)
F_{1,s}\bigl(
g_\ell(\bfw)u(\theta)
\widetilde{g_s(\bfv)}u(\phi)x
\bigr)
\,d\nu(\theta)
\,d\nu^{(g_s(\bfv))}(\phi).
\end{align*}
By the definition of $\|F_0\|_{\bfv,s}$,
\[
    \left|
    F_0\bigl(
        \phi+(A^{(g_s(\bfv))})^{-1}\theta
    \bigr)
    -
    F_0(\phi)
    \right|
    \leq
    \|F_0\|_{\bfv,s}.
\]
Therefore,
\begin{align}
\label{eq:inductive decomposition}
&\int_{\R^m}
F_0(\theta)
F_1\bigl(
g_s(\bfv)g_\ell(\bfw)u(\theta)x
\bigr)
\,d\nu(\theta)
\nonumber\\
&=
\int_{\R^m}
\bigl(
F_0(\phi)+O(\|F_0\|_{\bfv,s})
\bigr)
\left(
\int_{\R^m}
F_{1,s}\bigl(
g_\ell(\bfw)u(\theta)
\widetilde{g_s(\bfv)}u(\phi)x
\bigr)
\,d\nu(\theta)
\right)
\,d\nu^{(g_s(\bfv))}(\phi) \nonumber \\
&= \int_{\R^m}
F_0(\phi)
\left(
\int_{\R^m}
F_{1,s}\bigl(
g_\ell(\bfw)u(\theta)
\widetilde{g_s(\bfv)}u(\phi)x
\bigr)
\,d\nu(\theta)
\right)
\,d\nu^{(g_s(\bfv))}(\phi) +O(\|F_0\|_{\bfv,s} \|F_1\|_{C^0}).
\end{align}

 Let $\varsigma>0$, to be chosen later, and split the outer integral
according to the injectivity radius.
For $\phi\in J(\bfv,s,\varsigma,x)$, we have
\[
    \inj\bigl(
        \widetilde{g_s(\bfv)}u(\phi)x
    \bigr)
    >
    \varsigma.
\]
Hence, applying~\eqref{eq:e3} together with~\eqref{eq:e1} and
\eqref{eq:e2}, we obtain
\begin{align*}
&\int_{\R^m}
F_{1,s}\bigl(
g_\ell(\bfw)u(\theta)
\widetilde{g_s(\bfv)}u(\phi)x
\bigr)
\,d\nu(\theta)=
\mu_{\X}(F_1)
+
O\!\left(
e^{-\delta_{\bfw}\ell}
\varsigma^{-c_{\bfw}}
\|F_1\|_{C^{l_0}}
\right).
\end{align*}
On the other hand, for $\phi\in J(\bfv,s,\varsigma,x)^c$, we use the
trivial bound
\[
\int_{\R^m}
F_{1,s}\bigl(
g_\ell(\bfw)u(\theta)
\widetilde{g_s(\bfv)}u(\phi)x
\bigr)
\,d\nu(\theta)
=
\mu_{\X}(F_1)
+
O(\|F_1\|_{C^0}).
\]
It follows from~\eqref{eq:inductive decomposition} that
\begin{align*}
\int_{\R^m}
F_0(\theta)
F_1\bigl(
g_s(\bfv)g_\ell(\bfw)u(\theta)x
\bigr)
\,d\nu(\theta)
&=
\mu_{\X}(F_1)
\left(
\int_{\R^m}F_0(\phi)
\,d\nu^{(g_s(\bfv))}(\phi)
\right) +
O(\|F_0\|_{\bfv,s} \|F_1\|_{C^0})
\\
&\quad+
O\!\left(
e^{-\delta_{\bfw}\ell}
\varsigma^{-c_{\bfw}}
\|F_1\|_{C^{l_0}}\|F_0\|_{C^0}
\right)
\\
&\quad+
O\!\left(
\|F_0\|_{C^0}
\|F_1\|_{C^0}
\nu^{(g_s(\bfv))}
\bigl(J(\bfv,s,\varsigma,x)^c\bigr)
\right).
\end{align*}
By Proposition~\ref{prop:height function},
\[
\nu^{(g_s(\bfv))}
\bigl(J(\bfv,s,\varsigma,x)^c\bigr)
\ll
\varsigma^{\kappa_1}
\left(
1+
e^{-\kappa_3s}
\inj(x)^{-\kappa_2}
\right).
\]
Since $\inj(x)$ is bounded above on $\X$, this gives
\[
\nu^{(g_s(\bfv))}
\bigl(J(\bfv,s,\varsigma,x)^c\bigr)
\ll
\varsigma^{\kappa_1}
\left(
1+
\inj(x)^{-\kappa_2}
\right)
\ll
\varsigma^{\kappa_1}
\inj(x)^{-\kappa_2}.
\]
Moreover, by~\eqref{eq:c2} and the definition of $\|F_0\|_{\bfv,s}$,
\[
\int_{\R^m}
F_0(\phi)
\,d\nu^{(g_s(\bfv))}(\phi)
=
\nu(F_0)
+
O(\|F_0\|_{\bfv,s}).
\]
Combining the preceding estimates and using $\|F_1\|_{C^0}
    \ll
    \|F_1\|_{C^{l_0}}$,
we obtain
\begin{align*}
\int_{\R^m}
F_0(\theta)
F_1\bigl(
g_s(\bfv)g_\ell(\bfw)u(\theta)x
\bigr)
\,d\nu(\theta)
&=
\nu(F_0)
\mu_{\X}(F_1)
+
O\!\left(
\|F_0\|_{\bfv,s}\|F_1\|_{C^0}
\right)
\\
&\quad+
O\!\left(
\left(
e^{-\delta_{\bfw}\ell}\varsigma^{-c_{\bfw}}
+
\varsigma^{\kappa_1}\inj(x)^{-\kappa_2}
\right)
\|F_1\|_{C^{l_0}}\|F_0\|_{C^0}
\right).
\end{align*}
 Choosing $\varsigma$ so as to balance the two error terms
yields~\eqref{eq:dd1} for suitable $\delta_\alpha,c_\alpha>0$, and proves the
proposition.
\end{proof}

\begin{prop}
\label{prop:change of weights}
Suppose that there exist a weight $\bfw$, constants
$\delta_{\bfw},c_{\bfw}>0$ and an integer $l_0\geq1$ such
that~\eqref{eq:e3} holds for every $f\in C_c^\infty(\X)$, $x\in\X$, and
$t\geq0$.

Then, for every $\alpha\in(0,1)$, there exist constants
$\delta_\alpha,c_\alpha>0$ such that, for every
$\bfv\in[\alpha,1]^m$, $f\in C_c^\infty(\X)$, $x\in\X$, and $t\geq0$,
we have
\begin{align}
\label{eq:dd2}
     \int_{\R^m}
    f\bigl(g_t(\bfv)u(\theta)x\bigr)
    \,d\nu(\theta)
    =
    \mu_{\X}(f)
    +
    O\!\left(
        e^{-\delta_\alpha t}
        \inj(x)^{-c_\alpha}
        \|f\|_{C^{l_0}}
    \right).
\end{align}
\end{prop}
\begin{proof}
Fix $\alpha\in(0,1)$ and let $\bfv\in[\alpha,1]^m$. Set
\[
    \varrho
    =
    \frac12\min_{1\leq i\leq m} v_i,
    \qquad
    \ell=\varrho t,
    \qquad
    s=(1-\varrho)t,
\]
and define
\[
    v_i'
    =
    \frac{v_i-\varrho w_i}{1-\varrho},
    \qquad
    \bfv'=(v_1',\ldots,v_m').
\]
Since $\varrho \in (0,1)$, $w_i \leq 1$ and $\varrho\leq v_i/2$, we have
\[
    v_i'
    \geq v_i-\varrho w_i
    \geq v_i-\varrho
    \geq \frac{v_i}{2}
    \geq\frac{\alpha}{2}.
\]
Moreover since
\[
\sum_i v_i'=1,
\]
we get that $\bfv'$ is a weight and belongs to the set $[\alpha/2,1]^m$.

By construction,
\[
    g_t(\bfv)
    =
    g_\ell(\bfw)g_s(\bfv').
\]

We now apply Proposition~\ref{prop: inductive} with
$\bfv'$ in place of $\bfv$, with $F_0\equiv1$ and
$F_1=f$, and obtain
\begin{align*}
    \int_{\R^m}
    f\bigl(g_t(\bfv)u(\theta)x\bigr)\,d\nu(\theta)
    &=
    \int_{\R^m}
    F_0(\theta)
    F_1\bigl(g_s(\bfv')g_\ell(\bfw)u(\theta)x\bigr)
    \,d\nu(\theta)\\
    &=
    \mu_{\X}(f)
    +
    O\!\left(
        e^{-\delta_{\alpha/2}\ell}
        \inj(x)^{-c_{\alpha/2}}
        \|f\|_{C^{l_0}}
    \right).
\end{align*}
Since $\ell=\varrho t
    \geq {\alpha t}/{2}$, estimate~\eqref{eq:dd2} follows for a suitable choice of $\delta$ and $c$. This proves the proposition.
\end{proof}

\smallskip

\begin{proof}[Proof of Theorem~\ref{thm:Double equi}]
Fix $\alpha\in(0,1)$ and weights
$\bfv,\bfv'\in[\alpha,1]^m$.

Fix $t\geq r\geq0$ and set
\[
    \tau=\rho((t,\bfv),(r,\bfv')).
\]
If $\tau\leq0$, the claimed estimate is immediate after enlarging the
implied constant. We may therefore assume that $\tau>0$. Set
\[
    \ell=\frac{\tau}{4},
    \qquad
    s=t-\ell.
\]
Let $F_1,F_2\in C_c^\infty(\X)$ and $x_1,x_2\in\X$.

Define $F_0:\R^m\to\R$ by
\[
    F_0(\theta)
    =
    F_2\bigl(g_r(\bfv')u(\theta)x_2\bigr).
\]
 Conjugation by $g_r(\bfv')$ multiplies the $i$-th coordinate of
$u(\cdot)$ by $e^{r(1+v_i')}$, while the $i$-th entry of
$(A^{(g_s(\bfv))})^{-1}$ is $\ll e^{-s(1+v_i)}$, since
$g_s(\bfv)\widetilde{g_s(\bfv)}^{-1}$ stays within a fixed compact set.
Together with the
inequalities $r+rv_i'-(t+tv_i)\leq-\tau$ and $1+v_i\leq2$, valid for every $i$,
this gives
\begin{align*}
    \|F_0\|_{\bfv,s}
    &\ll
    \max_{1\leq i\leq m}
    e^{(r+rv_i')-(s+sv_i)}
    \|F_2\|_{C^1}
    =
    \max_{1\leq i\leq m}
    e^{(r+rv_i')-(t+tv_i)+\ell(1+v_i)}
    \|F_2\|_{C^1}
    \\
    &\ll
    e^{-\tau+2\ell}\|F_2\|_{C^1}
    =
    e^{-\tau/2}\|F_2\|_{C^1}.
\end{align*}

Note first that, by Proposition~\ref{prop:change of weights}, the
hypothesis~\eqref{eq:e3} of Proposition~\ref{prop: inductive} is satisfied by \emph{every}
weight in $[\alpha,1]^m$, and not merely by the distinguished weight $\bfw$: indeed the
estimate~\eqref{eq:dd2} is precisely~\eqref{eq:e3} with $\bfw$ replaced by
$\bfv$, with constants $\delta_\alpha$ and $c_\alpha$ that depend only on $\alpha$ and
not on $\bfv$. Proposition~\ref{prop: inductive} may therefore be applied with
$\bfv$ in place of $\bfw$, and the constants that it produces again depend only on
$\alpha$.

We now apply Proposition~\ref{prop: inductive} with
$\bfw=\bfv$, $F_0$ as above, and $F_1$. Since
\[
    g_t(\bfv)=g_s(\bfv)g_\ell(\bfv),
\]
we obtain
\begin{align*}
\int_{\R^m}
F_1\bigl(g_t(\bfv)u(\theta)x_1\bigr)
&F_2\bigl(g_r(\bfv')u(\theta)x_2\bigr)
\,d\nu(\theta)
=
\int_{\R^m}
F_0(\theta)
F_1\bigl(g_s(\bfv)g_\ell(\bfv)u(\theta)x_1\bigr)
\,d\nu(\theta)
\\
&=
\nu(F_0) \mu_{\X}(F_1)
+
O\!\left(\|F_0\|_{\bfv,s}\|F_1\|_{C^0}
+
e^{-\delta_\alpha \ell}
\inj(x_1)^{-c_\alpha}
\|F_1\|_{C^{l_0}}\|F_0\|_{C^0}
\right) \\
&= \nu(F_0)\mu_{\X}(F_1)
+
O\!\left(
(e^{-\tau/2}
 +
e^{-\delta_\alpha\tau/4})
\inj(x_1)^{-c_\alpha}
\|F_1\|_{C^{l_0}} \|F_2\|_{C^1}
\right),
\end{align*}
where the last inequality follows by substituting the bounds for
$\|F_0\|_{\bfv,s}$ and $\ell$, and using that $\inj(x)$ is uniformly bounded
above on $\X$.

Finally, applying Proposition~\ref{prop:change of weights} to the
$\bfv'$-weighted integral gives
\[
\nu(F_0)=\int_{\R^m}
F_2\bigl(g_r(\bfv')u(\theta)x_2\bigr)\,d\nu(\theta)
=
\mu_{\X}(F_2)
+
O\!\left(
e^{-\delta_\alpha r}
\inj(x_2)^{-c_\alpha}\|F_2\|_{C^{l_0}}
\right).
\]
Multiplying the last display by $\mu_{\X}(F_1)$ and inserting the
result into the preceding one yields the claimed estimate~\eqref{eq:dd3} for all
sufficiently small $\delta$ and all sufficiently large $c$; the factor $|\mu_{\X}(F_1)|$ in the last error
term of~\eqref{eq:dd3} arises in this way.
\end{proof}

\medskip

\section{The convergence case}
\label{sec:con case}

This section is devoted to the proof of the following proposition, which is
the weighted analogue of the convergence theorem
of~\cite[\S9]{KhalilLuethi}.

\begin{prop}
\label{prop:convergence case}
Let $\alpha\in(0,1)$ and assume that
$\Psi=(\psi_1,\ldots,\psi_m)$ is an
$m$-tuple of non-increasing functions $\psi_i:(0,\infty)\to(0,1)$
which is $\alpha$-balanced and satisfies
\begin{align}
\label{eq:assum con case}
    \sum_{q\in\N}\psi_1(q)\cdots\psi_m(q)<\infty.
\end{align}

Suppose that $\nu$ is a measure on $\R^m$ (not necessarily a product of
self-similar measures) such that there exist $\delta>0$ and $l_0\in\N$ with the following property: for every weight $\bfw\in[\alpha/2,1]^m$, every $f\in C_c^\infty(\X)$ and every $t\geq 0$,
\begin{align}
\label{eq:e5}
\int_{\R^m}
f(g_t(\bfw)u(\theta)\Gamma)\,d\nu(\theta)
=
\mu_{\X}(f)
+
O\!\left(
e^{-\delta t}
\|f\|_{C^{l_0}}
\right).
\end{align}
Then, for $\nu$-almost every $\theta$, there are only finitely many $(p,q)\in\Z^m\times\N$ satisfying $|p_i+q\theta_i|\leq\psi_i(q)$ for $1\leq i\leq m$.
\end{prop}
\begin{proof}
Fix $\nu$, $\Psi$, and $\alpha$ as in the statement, and write
\[
\psi(q)=\psi_1(q)\cdots\psi_m(q).
\]
Note that $\psi(q)<1$ for every $q$, since each $\psi_i$ takes
values in $(0,1)$.

For $n\geq1$, let $A_n$ be the set of $\theta\in\R^m$ for which there exists
$(p,q)\in\Z^m\times\N$ such that
\[
1\leq q\leq2^{n+1}
\quad\text{and}\quad
|p_i+\theta_iq|\leq\psi_i(2^n)
\quad\text{for all }1\leq i\leq m.
\]
By the monotonicity of the functions $\psi_i$, it suffices to show that
\[
\nu\left(\limsup_{n\to\infty}A_n\right)=0.
\]
By the Borel--Cantelli lemma, it is further enough to prove that
\[
\sum_{n\geq1}\nu(A_n)<\infty,
\]
which we now prove.

Define
\[
r_n=\left(2^{n+1}\psi(2^n)\right)^{1/(m+1)}.
\]
Since $\psi$ is non-increasing,
\begin{align}
\label{eq:f1}
\sum_{n\geq1}r_n^{m+1}
=\sum_{n\geq1}2^{n+1}\psi(2^n) \ll\sum_{n\geq1}\sum_{q=2^{n-1}}^{2^n}\psi(2^n) \leq\sum_{n\geq1}\sum_{q=2^{n-1}}^{2^n}\psi(q)
\ll\sum_{q\geq1}\psi(q)<\infty.
\end{align}
Thus, after discarding finitely many $n$, we may assume that $r_n\leq1$.

We consider two cases.

\medskip

\noindent\textbf{Case 1.}
Suppose that
\begin{align}
    \label{eq:con case assum 2}
    \psi(2^n)^{2-(2m+1)\alpha}
2^{(2-m\alpha)(n+1)}\geq1.
\end{align}
Define
\[
a_n=
\operatorname{diag}\left(
\psi_1(2^n)^{-1}r_n,\ldots,
\psi_m(2^n)^{-1}r_n,
2^{-(n+1)}r_n
\right).
\]
Then,
\begin{align*}
    A_n \subset
\left\{
    \theta:
    a_nu(\theta)\Z^{m+1}
    \cap[-r_n,r_n]^{m+1}\neq\{0\}
\right\}.
\end{align*}
By Lemma~\ref{lem:dani}, applied with $z_i=\psi_i(2^n)$ and
$Q=2^{n+1}$ --- its two hypotheses being $\alpha$-balance
and~\eqref{eq:con case assum 2} --- we obtain
\[
a_n=g_{t_n}(\bfv_n)
\]
for some $\bfv_n\in[\alpha/2,1]^m$ and
\[
t_n=\log\left(2^{n+1}r_n^{-1}\right)\geq n\log2.
\]

Let $f_n\in C_c^\infty(\X)$ be such that $0\leq f_n\leq1$,
\[
f_n(g\Gamma)=1
\quad\text{whenever}\quad
g\Z^{m+1}\cap[-r_n,r_n]^{m+1}\neq\{0\},
\]
and
\[
\operatorname{supp}(f_n)
\subset
\left\{
    g\Gamma\in\X:
    g  \Z^{m+1}\cap[-2r_n,2r_n]^{m+1}\neq\{0\}
\right\},
\]
with
\[
\|f_n\|_{C^{l_0}}\ll1.
\]
Such a choice is possible by a standard smoothing argument, with all implied
constants independent of $n$.

Then, using \eqref{eq:e5}, we get that
\begin{align*}
        \nu(A_n)
    \leq
    \int_{\R^m}
    f_n(a_nu(\theta)\Gamma)\,d\nu(\theta)
    =
    \mu_{\X}(f_n)
    +O\left(e^{-\delta t_n}\|f_n\|_{C^{l_0}}\right)
    \ll
    \mu_{\X}(f_n)+2^{-\delta n }.
\end{align*}
By the support condition on $f_n$ and Proposition~\ref{prop:roger},
\[
\mu_{\X}(f_n)
\ll
\mu_{\X}\left(
    \widehat{\ind}_{[-2r_n,2r_n]^{m+1}}
\right)
\ll r_n^{m+1}.
\]
Hence
\begin{equation}
    \nu(A_n)\ll r_n^{m+1}+2^{-\delta n }.
    \label{eq:imp 1}
\end{equation}

\medskip

\noindent\textbf{Case 2.}
Suppose that
\[
\psi(2^n)^{2-(2m+1)\alpha}
2^{(2-m\alpha)(n+1)}<1.
\]
By Remark~\ref{rem:balanced}, $\alpha\leq1/m$, so that
$2-m\alpha\geq1>0$; since $\psi(2^n)<1$, this forces
$2-(2m+1)\alpha>0$. Thus,
\[
B_n:=
2^{-\frac{(2-m\alpha)(n+1)}
{2-(2m+1)\alpha}} > \psi(2^n).
\]
Define
\[
w_{n,i}=
\frac{\log\psi_i(2^n)}
{\log\psi(2^n)}, \qquad s_n= (2^{n+1}B_n)^{1/(m+1)}.
\]
Since $\Psi$ is $\alpha$-balanced,
\[
w_{n,i}\in[\alpha,1],
\qquad
\sum_{i=1}^mw_{n,i}=1.
\]
Set
\[
a_n=
\operatorname{diag}\left(
B_n^{-w_{n,1}}s_n,\ldots,
B_n^{-w_{n,m}}s_n,
2^{-(n+1)}s_n
\right).
\]
By explicit computation, we note that $a_n= g_{t_n}(\bfv_n)$ for some weight $\bfv_n \in [\alpha/2,1]^m$ and
\[
t_n= \frac{(n+1)\log 2}{m+1}\left({m} + \frac{(2-m\alpha)}
{2-(2m+1)\alpha} \right) \geq n\log 2.
\]
Indeed, set
\[
    L_n=(n+1)\log 2,
    \qquad
    \kappa=\frac{2-m\alpha}{2-(2m+1)\alpha},
\]
so that $B_n=e^{-\kappa L_n}$ and
$s_n=\bigl(e^{L_n}B_n\bigr)^{1/(m+1)}=e^{(1-\kappa)L_n/(m+1)}$.
Since $2-(2m+1)\alpha>0$ in the present case and
$(2-m\alpha)-(2-(2m+1)\alpha)=(m+1)\alpha>0$, we have $\kappa>1$. The last diagonal
entry of $a_n$ equals
\[
    2^{-(n+1)}s_n
    =
    e^{-\frac{(m+\kappa)L_n}{m+1}},
\]
which identifies $t_n=\frac{(m+\kappa)L_n}{m+1}$, as displayed above; in particular
$t_n>L_n\geq n\log2$ because $\kappa>1$. For $1\leq i\leq m$, the $i$-th diagonal
entry equals
\[
    B_n^{-w_{n,i}}s_n
    =
    e^{\left(\kappa w_{n,i}+\frac{1-\kappa}{m+1}\right)L_n}
    =
    e^{t_nv_{n,i}},
    \qquad\text{where}\qquad
    v_{n,i}
    =
    \frac{(m+1)\kappa w_{n,i}+1-\kappa}{m+\kappa}.
\]
As $\sum_{i=1}^mw_{n,i}=1$, we get
$\sum_{i=1}^mv_{n,i}=\frac{(m+1)\kappa+m(1-\kappa)}{m+\kappa}=1$, so that $\bfv_n$ is
indeed a weight. Finally, $v_{n,i}$ is increasing in $w_{n,i}$ since $\kappa>0$,
 and a direct computation, using the definition of $\kappa$, shows that it takes
the value $\alpha/2$ \emph{exactly} at $w_{n,i}=\alpha$. Since $w_{n,i}\geq\alpha$, this gives
$v_{n,i}\geq\alpha/2$, and $v_{n,i}\leq1$ follows from
$\sum_{i=1}^mv_{n,i}=1$. Hence $\bfv_n\in[\alpha/2,1]^m$.

Moreover,
\begin{align}
    \label{eq:ee1}
A_n \subset \{\theta: a_nu(\theta)\Z^{m+1}  \cap[-s_n,s_n]^{m+1}\neq\{0\} \}.
\end{align}
This is because if $\theta \in A_n$, then there exists
$(p,q)\in\Z^m\times\N$ such that
\[
1\leq q\leq2^{n+1}
\quad\text{and}\quad
|p_i+\theta_iq|\leq\psi_i(2^n) = \psi(2^n)^{w_{n,i}}
\quad\text{for all }1\leq i\leq m.
\]
Hence there exists
$(p,q)\in\Z^m\times\N$ such that
\[
1\leq q\leq2^{n+1}
\quad\text{and}\quad
|p_i+\theta_iq|\leq B_n^{w_{n,i}}
\quad\text{for all }1\leq i\leq m,
\]
which immediately implies that $\theta$ belongs to the right-hand side of~\eqref{eq:ee1}.

 The measure of the right-hand side of~\eqref{eq:ee1} is computed as in
Case~1, by using~\eqref{eq:e5} and constructing a smooth function $f_n$. This
gives that
\begin{equation}
    \nu(A_n)\ll s_n^{m+1}+2^{-\delta n}.
    \label{eq:imp 2}
\end{equation}

 Combining~\eqref{eq:f1},~\eqref{eq:imp 1} and~\eqref{eq:imp 2}, and noting that
$s_n^{m+1}=2^{n+1}B_n=2^{(1-\kappa)(n+1)}$ with $\kappa>1$, so that the middle
series below is geometric, we obtain
\[
\sum_{n\geq1}\nu(A_n)
\ll
\sum_{n\geq1}
\left(
r_n^{m+1}+s_n^{m+1}+2^{-\delta n}
\right)
<\infty.
\]
Hence the proposition follows.
\end{proof}

\medskip

\section{The divergence case}
\label{sec:div case}
This section is devoted to the proof of the following proposition;
its unweighted counterpart is~\cite[Thm.~9.3]{benard2026}.

\begin{prop}
\label{prop:div case}
Assume that $m\geq2$, and let $\alpha\in(0,1)$ and $\beta>1$.
Let $\Psi=(\psi_1,\ldots,\psi_m)$ be an $m$-tuple of non-increasing functions
$\psi_i:(0,\infty)\to(0,1)$
which is $\alpha$-balanced and
$\beta$-proportional, and suppose that
\begin{align}
\label{eq:assum div case}
    \sum_{q\in\N}
    \psi_1(q)\cdots\psi_m(q)
    =\infty.
\end{align}
Suppose that $\nu$ is a measure on $\R^m$ whose
coordinate projections are non-atomic, that is,
\[
    \nu\bigl(\{\theta\in\R^m:\ \theta_i=c\}\bigr)=0
    \qquad\text{for all } 1\leq i\leq m,\ c\in\R,
\]
and that there exist $\delta>0$ and
$l_0\in\N$ with the following two properties.  First,~\eqref{eq:e5} holds
for every $\bfw\in[\alpha/2,1]^m$, every $f\in C_c^\infty(\X)$ and every
$t\geq0$. Second, for every pair of weight
vectors $\bfw,\bfv\in[\alpha/2,1]^m$, every
$f_1,f_2\in C_c^\infty(\X)$, and every $t\geq r\geq0$, the estimate
\begin{align}
\label{eq:dd3 3}
\int_{\R^m}
&f_1(g_t(\bfw)u(\theta)\Gamma)
 f_2(g_r(\bfv)u(\theta)\Gamma)
\,d\nu(\theta) \nonumber\\
&=
\mu_{\X}(f_1)\mu_{\X}(f_2)
+
O\left(
e^{-\delta \rho((t,\bfw),(r,\bfv))}
\|f_1\|_{C^{l_0}}\|f_2\|_{C^{l_0}}
\right) \nonumber\\
&\quad+
O\left(
e^{-\delta r}
\bigl|\mu_{\X}(f_1)\bigr|\|f_2\|_{C^{l_0}}
\right),
\end{align}
holds, where $\rho(\cdot,\cdot)$ is defined as in~\eqref{eq:def rho}.

Then, for $\nu$-almost every $\theta$, we have
\begin{align}
\label{eq:main thm asym 2}
    \cN_T(\theta,\Psi)
    \mathrel{\sim} 2^m
    \sum_{q=1}^T
    \psi_1(q)\cdots\psi_m(q),
\end{align}
as $T\to\infty$.
\end{prop}

The proof follows along the lines of that result. We briefly explain the strategy and the
necessary modifications. Before starting the proof, as in the proof of
Proposition~\ref{prop:convergence case}, we define
\[
    \psi(q)=\psi_1(q)\cdots\psi_m(q),
    \qquad q\in[1,\infty).
\]

The proof occupies the whole of this section and is organised as
follows. We first reduce the two-sided count $\cN_T(\theta,\Psi)$ to a
one-sided one, and fix a dyadic-type parameter $\tau>1$; the counting function
is then split into the blocks $\rS_k(\theta)$ (for the lower bound) and
$\rS_k^+(\theta)$ (for the upper bound), each of which is
the Siegel transform of the indicator of a box, evaluated at
$a_ku(\theta)\Gamma$ for a diagonal element $a_k$. The indices $k$ are then split
into three families: $K_{\mathrm{tiny}}$, whose total contribution is finite
and therefore negligible; $K_{\mathrm{big}}$, where the box is so large that a
deterministic lattice point count \cite[Lem.~9.6]{benard2026} suffices; and
$K_{\mathrm{small}}$, where one passes to smoothed Siegel transforms and
applies a second moment estimate together with a Borel--Cantelli argument. The
lower bound is carried out in Section~\ref{subsec:div lower} and the upper
bound, which follows the same lines, in Section~\ref{subsec:div upper}.

We first reduce Proposition~\ref{prop:div case} to a one-sided counting
statement. For $\boldsymbol{\omega}=(\omega_1,\ldots,\omega_m)\in\{\pm1\}^m$ and
$T>0$ put
\[
    \cN^{\boldsymbol{\omega}}_T(\theta,\Psi)
    =
    \#\bigl\{(p,q)\in\Z^m\times\N:\ q\leq T,\
    0<\omega_i(p_i+\theta_iq)<\psi_i(q)\ \text{ for all } i\bigr\}.
\]

Then, for $\nu$-almost every $\theta\in\R^m$ and every $T>0$,
\begin{align}
\label{eq:orthant decomposition}
    \cN_T(\theta,\Psi)
    =
    \sum_{\boldsymbol{\omega}\in\{\pm1\}^m}\cN^{\boldsymbol{\omega}}_T(\theta,\Psi).
\end{align}
Indeed, the solutions not accounted for on the right-hand side are those with
$p_i+\theta_iq=0$ or $|p_i+\theta_iq|=\psi_i(q)$ for some $i$, and for fixed
$(i,p_i,q)$ each of these conditions confines $\theta_i$ to at most two points.
The exceptional set is therefore contained in a countable union of
hyperplanes $\{\theta_i=c\}$, which is $\nu$-null because the coordinate
projections of $\nu$ are non-atomic.

It therefore suffices to prove that, for each fixed
$\boldsymbol{\omega}\in\{\pm1\}^m$ and for $\nu$-almost every $\theta$,
\begin{align}
\label{eq:one sided asym}
    \cN^{\boldsymbol{\omega}}_T(\theta,\Psi)
    \ \sim\
    \sum_{q=1}^{T}\psi_1(q)\cdots\psi_m(q),
    \qquad T\to\infty .
\end{align}
For simplicity of notation we treat only the case
$\boldsymbol{\omega}=(1,\ldots,1)$; the remaining $2^m-1$ cases are proved
in exactly the same way, since no step below uses the position of the
boxes introduced below,
but only their side lengths and their Lebesgue
measure. Thus $\boldsymbol{\omega}$ disappears from the notation from now
on, and we write
\[
    \cN^+_T(\theta,\Psi)
    =
    \#\bigl\{(p,q)\in\Z^m\times\N:\ q\leq T,\
    0<p_i+\theta_iq<\psi_i(q)\ \text{ for all } i\bigr\} .
\]

\subsection{Lower bound}
\label{subsec:div lower}

For $k\in\N$ and $\tau>1$, define
\[
\rS_{\tau,k}(\theta)
=
\#\left\{
(p_1,\ldots,p_m,q)\in\Z^m\times\N:
\begin{array}{l}
q\in[\tau^{k-1},\tau^k),\\
0<p_i+\theta_iq<\psi_i(\tau^k)
\text{ for all }1\leq i\leq m
\end{array}
\right\}.
\]
The half-open interval $[\tau^{k-1},\tau^k)$ is used so that the ranges of $q$
corresponding to distinct $k$ are disjoint.
\begin{lem}
\label{lem:enough for power series}
Assume the hypotheses of Proposition~\ref{prop:div case}. Then, in order to
prove the lower bound in~\eqref{eq:one sided asym}, it is enough to show that,
for every $\tau\in(1,2]$, for $\nu$-almost every
$\theta\in\R^m$, for every $\eta>0$, and for all sufficiently large
$n$,
\begin{align}
\label{eq:red 1 div case}
\sum_{k=1}^n \rS_{\tau,k}(\theta)
\geq
(1-\tau^{-1}-\eta)
\sum_{k=1}^n \psi(\tau^k)\tau^k.
\end{align}
\end{lem}

\begin{proof} 
The proof is analogous to the corresponding proof in~\cite{benard2026}. We include the details for completeness.

Write $S_\Psi(T)=\sum_{q\leq T}\psi(q)$. Let $\tau\in(1,2]$, let
$\theta$ be such that~\eqref{eq:red 1 div case} holds for every $\eta>0$ and
all large $n$, let $T$ be large and let $n$ be defined by
$\tau^n\leq T<\tau^{n+1}$.

Since the ranges of $q$ occurring in $\rS_{\tau,1},\ldots,\rS_{\tau,n}$ are
disjoint and contained in $[1,T]$, and since $\psi_i(\tau^k)\leq\psi_i(q)$ for
$q\leq\tau^k$, every pair counted by some $\rS_{\tau,k}$ with $k\leq n$ is
counted by $\cN^+_T(\theta,\Psi)$, and no pair is counted twice. Hence
\begin{align}
\label{eq:orthant lower 1}
    \cN^+_T(\theta,\Psi)\ \geq\ \sum_{k=1}^n\rS_{\tau,k}(\theta).
\end{align}

Next we compare $\sum_{k\leq n}\psi(\tau^k)\tau^k$ with $S_\Psi(T)$. Fix
$\varepsilon\in(0,\tfrac12)$. Since $\tau^{k+1}-\tau^k\to\infty$ and
$\bigl|\lceil\tau^{k+1}\rceil-\lceil\tau^{k}\rceil-(\tau^{k+1}-\tau^k)\bigr|\leq1$,
there is $k_0=k_0(\tau,\varepsilon)$ such that
\[
    (1-\varepsilon)(\tau^{k+1}-\tau^k)
    \ \leq\
    \lceil\tau^{k+1}\rceil-\lceil\tau^{k}\rceil
    \ \leq\
    (1-\varepsilon)^{-1}(\tau^{k+1}-\tau^k)
    \qquad(k\geq k_0).
\]
Therefore, $\psi$ being non-increasing,
\[
    (1-\tau^{-1})\psi(\tau^k)\tau^{k+1}
    =
    \psi(\tau^k)\bigl(\tau^{k+1}-\tau^k\bigr)
    \ \geq\
    (1-\varepsilon)\!\!\sum_{q=\lceil\tau^{k}\rceil}^{\lceil\tau^{k+1}\rceil-1}\!\!\psi(q) .
\]
Summing over $k_0\leq k\leq n$ and using $S_\Psi(T)\to\infty$ to absorb the
finitely many terms with $k<k_0$, we obtain, for $T$ large,
\begin{align}
\label{eq:orthant lower 2}
    (1-\tau^{-1})\,\tau\sum_{k=1}^{n}\psi(\tau^k)\tau^k
    \ \geq\
    (1-2\varepsilon)\,S_\Psi(T),
\end{align}
because $\lceil\tau^{n+1}\rceil-1\geq T$.

Combining~\eqref{eq:red 1 div case}, \eqref{eq:orthant lower 1}
and~\eqref{eq:orthant lower 2},
\[
    \liminf_{T\to\infty}\frac{\cN^+_T(\theta,\Psi)}{S_\Psi(T)}
    \ \geq\
    \frac{(1-2\varepsilon)\bigl(1-\tau^{-1}-\eta\bigr)}{(1-\tau^{-1})\,\tau} .
\]
The right-hand side tends to $1$ if we let $\varepsilon\to0$, then $\eta\to0$, and
finally $\tau\to1^+$. Choosing sequences $\tau_j\downarrow1$,
$\eta_j\downarrow0$ and $\varepsilon_j\downarrow0$ along which this happens and
intersecting the corresponding countably many sets of full $\nu$-measure, we
conclude that $\liminf_{T}\cN^+_T(\theta,\Psi)/S_\Psi(T)\geq1$ for
$\nu$-almost every $\theta$.
\end{proof}

For the remainder of this section, we fix $\tau>1$ and write
$\rS_k(\theta)$ in place of $\rS_{\tau,k}(\theta)$. To proceed further,
we define
\begin{align}
    r_k
    &=
    \bigl(\psi(\tau^k)\tau^k\bigr)^{\frac{1}{m+1}},
    \label{eq:def r k}\\
    a_k
    &=
    \begin{pmatrix}
        \psi_1(\tau^k)^{-1}r_k & & & \\
        & \ddots & & \\
        & & \psi_m(\tau^k)^{-1}r_k & \\
        & & & \tau^{-k}r_k
    \end{pmatrix},
    \label{eq:def a k}\\
    R_k
    &=
    [0,r_k]^m\times[\tau^{-1}r_k,r_k]
    \subset\R^{m+1},\nonumber\\
    t_k
    &=
    \tau^k r_k^{-1}
    =
    (\tau^k)^{\frac{m}{m+1}}
    \bigl(\psi(\tau^k)\bigr)^{-\frac{1}{m+1}},\\
    s_k
    &=
    \log t_k.
\end{align}
Then
\[
    \rS_k(\theta)
    =
    \widehat{\ind}_{R_k}\bigl(a_k u(\theta)\Gamma\bigr),
\]
where $\Gamma$ denotes the identity coset in $\X$.

It will be convenient to record at once how this normalisation
compares with that of~\cite{benard2026}, from which several estimates are
quoted below. Writing $\psi_\ast$ for the approximation function
of~\cite{benard2026}, so that $\psi=\psi_\ast^{\,m}$, the quantity
$r_k$ is the same in both papers, whereas the parameter used there is
$\tau^k\psi_\ast(\tau^k)^{-1}$ and is related to $t_k=\tau^kr_k^{-1}$
by
\[
    t_k=\bigl(\tau^k\psi_\ast(\tau^k)^{-1}\bigr)^{\frac{m}{m+1}} .
\]
Accordingly, the statements of \cite[\S9.1--9.2]{benard2026} will be applied
here with their parameter $\gamma_1$ taken to be $\frac{m}{m+1}\gamma_1$; as
both are arbitrarily small, this affects nothing, but we record it to fix the
normalisation.

We stress that no argument below uses the position of the boxes
involved, but only their side lengths and their Lebesgue measure. This is what
allows us to treat only the orthant
$\boldsymbol{\omega}=(1,\ldots,1)$.

We also define $K_{\mathrm{tiny}}$ to be the set of all $k\in\N$ such that
\begin{align}
    \psi(\tau^k)^{2-(2m+1)\alpha}
    (\tau^k)^{2-m\alpha}
    <1,
    \label{eq:assum}
    \quad \text{or}\qquad
    \psi(\tau^k)\tau^k\log^2(\tau^k)
    < 1.
\end{align}
The second condition is the negation of the standing hypothesis
\cite[(9.18)]{benard2026} of \cite[Prop.~9.12]{benard2026}, which in the
normalisation fixed above reads $\psi(\tau^k)\tau^k\log^2(\tau^k)\geq1$.
Discarding $K_{\mathrm{tiny}}$ thus guarantees it for every
$k\in K_{\mathrm{small}}$, which is the reason for the factor
$\log^2(\tau^k)$.

Both hypotheses of Proposition~\ref{prop:div case} remain valid when
$\delta$ is decreased and $l_0$ is increased, and we shall use this
repeatedly without further mention.

We record the following lemma.
\begin{lem}
\label{lem:imp change}
The following properties hold.
\begin{enumerate}
    \item[(i)] The sequence $(t_k)_k$ is increasing. Moreover, there exists
    $\gamma>0$ such that, for all sufficiently large $k$,
    \[
        s_k\geq \gamma k,
        \qquad
        t_k\geq e^{\gamma k}.
    \]

    \item[(ii)] For each $k\in\mathbb N\setminus K_{\mathrm{tiny}}$, there exists a
    weight vector $\bfw_k\in[\alpha/2,1]^m$ such that
    \[
        a_k=g_{s_k}(\bfw_k).
    \]

    \item[(iii)] After decreasing $\delta$ and increasing $l_0$,
    for all
    $k\in\mathbb N\setminus K_{\mathrm{tiny}}$ and
    $f\in C_c^\infty(\X)$,
    \begin{align}
    \label{eq:mix 1}
        \int_{\mathbb R^m} f(a_ku(\theta)\Gamma)\,d\nu(\theta)
        =
        \mu_\X(f)
        +
        O\left(t_k^{-\delta}\|f\|_{C^{l_0}}\right).
    \end{align}
    Furthermore, for all $k>j$ with
    $k,j\in\mathbb N\setminus K_{\mathrm{tiny}}$ and
    $f_1,f_2\in C_c^\infty(\X)$, we have
    \begin{align}
    \label{eq:mix 2}
        &\int_{\mathbb R^m}
        f_1(a_ku(\theta)\Gamma)
        f_2(a_ju(\theta)\Gamma)\,d\nu(\theta)
        =
        \mu_\X(f_1)\mu_\X(f_2)
        +O\left(
            t_k^{-\delta}
            \|f_1\|_{C^{l_0}}\|f_2\|_{C^{l_0}}
        \right) \nonumber\\
        &\qquad
        +O\left(
            t_j^{-\delta}
            \|f_2\|_{C^{l_0}}\bigl|\mu_\X(f_1)\bigr|
        \right)
        +O\left(
            t_k^{-\delta}t_j^\delta
            \mathcal{S}_{2,l_0}(f_1)
            \mathcal{S}_{2,l_0}(f_2)
        \right).
    \end{align}
\end{enumerate}
\end{lem}

\begin{proof}
We first prove (i). By definition,
\[
    s_k
    = k\log(\tau)-\log(r_k)
    = \frac{m}{m+1}k\log(\tau)
       -\frac{1}{m+1}\log(\psi(\tau^k)).
\]
Since $\psi$ is non-increasing, the sequence
$\bigl(\psi(\tau^k)\bigr)_k$ is non-increasing, and hence $(s_k)_k$ is
increasing. Moreover,
\[
    s_k-s_{k-1}
    =
    \frac{m}{m+1}\log(\tau)
    -\frac{1}{m+1}
      \log\left(
      \frac{\psi(\tau^k)}{\psi(\tau^{k-1})}
      \right)
    \geq
    \frac{m}{m+1}\log(\tau).
\]
Consequently,
\[
    s_k
    \geq
    s_1+\frac{m}{m+1}(k-1)\log(\tau),
\]
and therefore, for some $\gamma>0$ and all sufficiently large $k$,
\[
    s_k\geq \gamma k,
    \qquad
    t_k=e^{s_k}\geq e^{\gamma k}.
\]
This proves (i).

Part (ii) is Lemma~\ref{lem:dani}, applied with
$z_i=\psi_i(\tau^k)$ and $Q=\tau^k$: the first hypothesis of that lemma is
the $\alpha$-balance of $\Psi$, and the second is the negation of the first
condition in~\eqref{eq:assum}, which holds because
$k\notin K_{\mathrm{tiny}}$. Note that $r=(\psi(\tau^k)\tau^k)^{1/(m+1)}=r_k$
and $\log(Q/r)=\log t_k=s_k$, in accordance
with~\eqref{eq:def r k}.

We now prove (iii).  By part~(ii), the estimate~\eqref{eq:mix 1} is
precisely~\eqref{eq:e5} applied to the weight $\bfw_k$ at time $s_k$, since
$t_k^{-\delta}=e^{-\delta s_k}$. It remains to prove~\eqref{eq:mix 2}.

Part (iii) is the weighted analogue of
\cite[Lem.~7.5, Cor.~7.6]{benard24}, see also \cite[(9.3)]{benard2026}, where
the corresponding estimate is established for the single flow $g_t(\bfw)$ with
$\bfw=(1/m,\ldots,1/m)$. The new feature here is that two different weights
$\bfw_k$ and $\bfw_j$ occur, so that the error term supplied by
\eqref{eq:dd3 3} is governed by $\rho((s_k,\bfw_k),(s_j,\bfw_j))$ rather than
by $s_k-s_j$; it is $\beta$-proportionality that converts the former into the
latter, as the computation below shows.

Fix $k>j$ with $k,j\notin K_{\mathrm{tiny}}$. We distinguish two cases
according to whether
\[
    t_k\geq t_j^{1+\varepsilon_0}
    \qquad\text{or}\qquad
    t_k\leq t_j^{1+\varepsilon_0},
\]
where the parameter $\varepsilon_0>0$ is fixed in Case~2
below.

\medskip
\noindent\textbf{Case 1: $t_k\geq t_j^{1+\varepsilon_0}$.}

In this case,
\[
    \frac{t_k}{t_j}
    \geq
    t_k^{1-\frac{1}{1+\varepsilon_0}}
    =
    t_k^{\frac{\varepsilon_0}{1+\varepsilon_0}}.
\]

By part (ii), we may write
\[
    a_k=g_{s_k}(\bfw_k),
    \qquad
    a_j=g_{s_j}(\bfw_j),
\]
with $\bfw_k,\bfw_j\in[\alpha/2,1]^m$. Therefore, by the hypothesis
\eqref{eq:dd3 3},
\begin{align*}
    \int_{\mathbb R^m}
    f_1(a_ku(\theta)\Gamma)
    f_2(a_ju(\theta)\Gamma)\,d\nu(\theta)
    &=
    \mu_\X(f_1)\mu_\X(f_2) \\
    &\quad+
    O\left(
        e^{-\delta
        \rho((s_k,\bfw_k),(s_j,\bfw_j))}
        \|f_1\|_{C^{l_0}}\|f_2\|_{C^{l_0}}
    \right)\\
    &\quad+
    O\left(
        e^{-\delta s_j}
        \|f_2\|_{C^{l_0}}\bigl|\mu_\X(f_1)\bigr|
    \right).
\end{align*}

By the definition of $\rho$,
\begin{align*}
    \rho((s_k,\bfw_k),(s_j,\bfw_j))
    &=
    \min_{1\leq i\leq m}
    \left\{
        s_k+s_kw_{k,i}
        -(s_j+s_jw_{j,i})
    \right\}\\
    &=
    \min_{1\leq i\leq m}
    \left\{
        \log(\tau^k\psi_i(\tau^k)^{-1})
        -
        \log(\tau^j\psi_i(\tau^j)^{-1})
    \right\}\\
    &=
    \log(\tau^{k-j})
    +
    \min_{1\leq i\leq m}
    \left\{
        \log\psi_i(\tau^j)-\log\psi_i(\tau^k)
    \right\}.
\end{align*}
 By $\beta$-proportionality and the monotonicity of the coordinate
functions,
\[
    \log\psi_i(\tau^j)-\log\psi_i(\tau^k)
    \gg
    \log\psi(\tau^j)-\log\psi(\tau^k)
    \geq0 ,
\]
the implied constant depending only on $m$ and $\beta$. Consequently,
\[
    \rho((s_k,\bfw_k),(s_j,\bfw_j))
    \gg
    \log(\tau^{k-j})+\log\psi(\tau^j)-\log\psi(\tau^k)
    \geq
    \log t_k-\log t_j
    \geq
    \frac{\varepsilon_0}{1+\varepsilon_0}\log t_k ,
\]
the second inequality because $\log t_k-\log t_j$ is the convex combination of
the two non-negative quantities $\log(\tau^{k-j})$ and
$\log\psi(\tau^j)-\log\psi(\tau^k)$ with coefficients $\tfrac{m}{m+1}$ and
$\tfrac{1}{m+1}$.

Since $t_j=e^{s_j}$, we obtain~\eqref{eq:mix 2} after decreasing
$\delta$ in this case.

\medskip
\noindent\textbf{Case 2: $t_k\leq t_j^{1+\varepsilon_0}$.}

In this case,
\[
    \frac{t_k}{t_j}
    \leq
    t_k^{\frac{\varepsilon_0}{1+\varepsilon_0}}.
\]
We first control the distortion of the diagonal matrix
$a_ja_k^{-1}$. Its first $m$ diagonal entries are
\[
    \frac{r_j}{r_k}
    \frac{\psi_i(\tau^k)}{\psi_i(\tau^j)},
    \qquad 1\leq i\leq m,
\]
while its last diagonal entry is $t_kt_j^{-1}$.

 Since $\psi_i$ is non-increasing,
\[
    0\leq
    \log\frac{\psi_i(\tau^j)}{\psi_i(\tau^k)}
    \leq
    \log\frac{\psi(\tau^j)}{\psi(\tau^k)} .
\]
Moreover $t_k/t_j=\tau^{k-j}r_j/r_k$ and $r_k^{m+1}=\tau^k\psi(\tau^k)$, whence
\[
    \log\frac{\psi(\tau^j)}{\psi(\tau^k)}
    =
    (m+1)\log\frac{r_j}{r_k}+\log\tau^{k-j}
    \ll
    \log\frac{t_k}{t_j},
    \qquad
    m\log\frac{r_k}{r_j}
    =
    \log\frac{t_k}{t_j}+\log\frac{\psi(\tau^k)}{\psi(\tau^j)}
    \ll
    \log\frac{t_k}{t_j},
\]
the implied constants depending only on $m$. Since
$\log(t_k/t_j)\leq\frac{\varepsilon_0}{1+\varepsilon_0}\log t_k$, it follows that
there exists $C=C(m)>0$ such that every diagonal entry of $a_ja_k^{-1}$ and of
its inverse is bounded above by $t_k^{C\varepsilon_0}$.

Consequently there exists $B=B(l_0,m)>0$ such that
\begin{align}
\label{eq:bb2}
    \|f_2\circ(a_ja_k^{-1})\|_{C^{l_0}}
    \ll
    \|f_2\|_{C^{l_0}}t_k^{B\varepsilon_0}.
\end{align}
Here we enlarge $l_0$ once and for all so that it exceeds the
integer $l'$ of Theorem~\ref{thm:mixing}, which will be applied below; $B$ is
then determined by $l_0$ and $m$, and we now fix $\varepsilon_0>0$ so small
that $B\varepsilon_0<\delta/2$. Since $\delta$ is decreased only at the very
end of the proof, this choice is legitimate in both cases.

Applying~\eqref{eq:mix 1} to the function
\[
    x\longmapsto f_1(x)f_2(a_ja_k^{-1}x),
\]
we obtain
\begin{align}
    &\int_{\mathbb R^m}
    f_1(a_ku(\theta)\Gamma)
    f_2(a_ju(\theta)\Gamma)\,d\nu(\theta)
    \nonumber\\
    &\qquad=
    \int_\X
    f_1(x)f_2(a_ja_k^{-1}x)\,d\mu_\X(x)
    +O\left(
        t_k^{-\delta}
        \|f_1\cdot(f_2\circ(a_ja_k^{-1}))\|_{C^{l_0}}
    \right). \label{eq:bb1}
\end{align}
 By~\eqref{eq:bb2} and the standard product estimate for the
$C^{l_0}$-norm, and since $B\varepsilon_0<\delta/2$, the error term in~\eqref{eq:bb1} is
\[
    O\bigl(
        t_k^{-\delta+B\varepsilon_0}
        \|f_1\|_{C^{l_0}}\|f_2\|_{C^{l_0}}
    \bigr)
    =
    O\bigl(
        t_k^{-\delta/2}
        \|f_1\|_{C^{l_0}}\|f_2\|_{C^{l_0}}
    \bigr).
\]

It remains to estimate the main term in~\eqref{eq:bb1}.
We apply the mixing
estimate of Theorem~\ref{thm:mixing} to $g=a_ja_k^{-1}$. By the definitions
\eqref{eq:def a k} of $a_k$ and of $t_k$, this is the diagonal matrix whose last entry equals
$t_kt_j^{-1}$, and hence
\[
    \bigl\|a_ja_k^{-1}\bigr\|\geq t_kt_j^{-1}>1 .
\]
Consequently, by Theorem~\ref{thm:mixing},
\begin{align*}
    \int_\X
    f_1(x)f_2(a_ja_k^{-1}x)\,d\mu_\X(x)
    &=
    \mu_\X(f_1)\mu_\X(f_2)
    +O\left(
        \left(t_kt_j^{-1}\right)^{-\delta_2}
        \mathcal{S}_{2,l_0}(f_1)
        \mathcal{S}_{2,l_0}(f_2)
    \right).
\end{align*}

Combining this with the preceding estimate, we obtain~\eqref{eq:mix 2} after
decreasing $\delta$, which completes the proof.
\end{proof}

We now proceed with the proof of~\eqref{eq:red 1 div case}. Fix
$\gamma_1\in(0,1/2)$, and further partition
$\N\setminus K_{\mathrm{tiny}}$ into the two sets
\[
    K_{\mathrm{big}}
    :=
    \left\{
        k\in\N\setminus K_{\mathrm{tiny}}:
        r_k\geq t_k^{\gamma_1}
    \right\}
\]
and
\[
    K_{\mathrm{small}}
    :=
    \left\{
        k\in\N\setminus K_{\mathrm{tiny}}:
        r_k<t_k^{\gamma_1}
    \right\}.
\]
The two parameters entering this partition are chosen in the following
order: first $\gamma_2>0$, small in terms of $m$ and $\alpha$ only, and then
$\gamma_1$, small compared with $\gamma_2$; both are fixed in the proof of
Lemma~\ref{lem:lower K small} below, and the partition above is the one they
determine.

\subsubsection{Lower asymptotics over $K_{\mathrm{big}}$}

We begin by proving the following lemma.

\begin{lem}
\label{lem:dist is log}
For $\nu$-almost every $\theta$ and all sufficiently large $k$, we have
\[
    \mathrm{dist}(a_k u(\theta)\Gamma,\Gamma)
    \ll \log\log t_k.
\]
\end{lem}

\begin{proof}
Using Lemma~\ref{lem:imp change} wherever necessary, the proof is
analogous to that of~\cite[Lem.~9.5]{benard2026}, and is therefore
omitted.
\end{proof}

We now prove the following lower bound.

\begin{lem}
\label{lem:lower K big}
Assume that
\begin{align}
\label{eq:assum K big div}
    \sum_{k\in K_{\mathrm{big}}}
    \psi(\tau^k)\tau^k=\infty.
\end{align}
Then, for $\nu$-almost every $\theta$ and every $\eta>0$, there exists
$n_\theta\in\N$ such that, for all $n\geq n_\theta$,
\begin{align}
\label{eq:sum over K big}
    \sum_{\substack{k\in K_{\mathrm{big}}\\ k\leq n}}
    \rS_k(\theta)
    \geq
    (1-\tau^{-1}-\eta)
    \sum_{\substack{k\in K_{\mathrm{big}}\\ k\leq n}}
    \psi(\tau^k)\tau^k.
\end{align}
\end{lem}

\begin{proof}
Fix $\theta$ in the full $\nu$-measure set given by
Lemma~\ref{lem:dist is log}. Then there exists $k_\theta\in\N$ such that,
for every $k\geq k_\theta$, we may write
\[
    a_k u(\theta)\Gamma=g\Gamma
\]
for some $g\in G$, depending on $k$, satisfying, with
$C_0=C_0(m)>0$ as in Lemma~\ref{lem:dist is log},
\[
    \|g\|
    \ll (\log t_k)^{C_0}
    \ll t_k^{\gamma_1/100}.
\]
Moreover, for $k\in K_{\mathrm{big}}$, the box $R_k$ has minimal side
length
\[
    (1-\tau^{-1})r_k
    \geq
    (1-\tau^{-1})t_k^{\gamma_1}.
\]
In particular the ratio of $\|g\|$ to the minimal side length of $R_k$ is
$\ll t_k^{\gamma_1/100-\gamma_1}\leq t_k^{-\gamma_1/2}$.
Therefore, by~\cite[Lem.~9.6]{benard2026}, for all sufficiently large
$k\in K_{\mathrm{big}}$,
\[
    \widehat{\ind}_{R_k}(a_k u(\theta)\Gamma)
    \geq
    \bigl(1-t_k^{-\gamma_1/4}\bigr)
    \operatorname{Leb}(R_k).
\]
Since
\[
    \operatorname{Leb}(R_k)
    =
    (1-\tau^{-1})\psi(\tau^k)\tau^k,
\]
we obtain
\[
    \rS_k(\theta)
    \geq
    \bigl(1-t_k^{-\gamma_1/4}\bigr)
    (1-\tau^{-1})
    \psi(\tau^k)\tau^k
\]
for all sufficiently large $k\in K_{\mathrm{big}}$. Consequently,
\begin{align*}
    \sum_{\substack{k\in K_{\mathrm{big}}\\ k_\theta\leq k\leq n}}
    \rS_k(\theta)
    \geq
    (1-\tau^{-1})
    \sum_{\substack{k\in K_{\mathrm{big}}\\ k_\theta\leq k\leq n}}
    \bigl(1-t_k^{-\gamma_1/4}\bigr)
    \psi(\tau^k)\tau^k.
\end{align*}
Since $t_k\to\infty$, the factor
$t_k^{-\gamma_1/4}$ tends to zero. Thus, for all sufficiently large $k$, say $k \geq k_\eta'$, we have
\[
    (1-\tau^{-1})
    \bigl(1-t_k^{-\gamma_1/4}\bigr)
    \geq
    1-\tau^{-1}-\frac{\eta}{2}.
\]
It follows that
\[
    \sum_{\substack{k\in K_{\mathrm{big}}\\\max\{k_\eta', k_\theta\}\leq k\leq n}}
    \rS_k(\theta)
    \geq
    \left(1-\tau^{-1}-\frac{\eta}{2}\right)
    \sum_{\substack{k\in K_{\mathrm{big}}\\ \max\{k_\eta', k_\theta\}\leq k\leq n}}
    \psi(\tau^k)\tau^k.
\]
Finally, the assumption~\eqref{eq:assum K big div} implies that the
finitely many indices $k<\max\{k_\eta', k_\theta\}$ contribute
negligibly compared with
\[
    \sum_{\substack{k\in K_{\mathrm{big}}\\ k\leq n}}
    \psi(\tau^k)\tau^k
\]
as $n\to\infty$. Hence, after increasing $n_\theta$ if necessary, we
obtain~\eqref{eq:sum over K big}.
\end{proof}

\subsubsection{Lower asymptotics over $K_{\mathrm{small}}$}

 We now prove the analogous lower bound over $K_{\mathrm{small}}$.

\begin{lem}
\label{lem:lower K small}
Assume that
\begin{align}
\label{eq:assum K small div}
    \sum_{k\in K_{\mathrm{small}}}
    \psi(\tau^k)\tau^k=\infty.
\end{align}
Then, for $\nu$-almost every $\theta$ and every $\eta>0$, there exists
$n_\theta\in\N$ such that, for all $n\geq n_\theta$,
\begin{align}
\label{eq:estimate K small}
    \sum_{\substack{k\in K_{\mathrm{small}}\\ k\leq n}}
    \rS_k(\theta)
    \geq
    (1-\tau^{-1}-\eta)
    \sum_{\substack{k\in K_{\mathrm{small}}\\ k\leq n}}
    \psi(\tau^k)\tau^k.
\end{align}
\end{lem}

\begin{proof}
We fix two parameters $\varepsilon,\gamma_2>0$, to be specified later,
and define
\[
    R_k^-
    =
    [\varepsilon r_k,(1-\varepsilon)r_k]^m
    \times
    [(\tau^{-1}+\varepsilon)r_k,(1-\varepsilon)r_k].
\]
Let $\chi_k:\X\to[0,1]$ be the function
\[
    \chi_k(x)
    =
    \begin{cases}
        1, & \text{if }\inj(x)\geq t_k^{-\gamma_2},\\
        0, & \text{otherwise}.
    \end{cases}
\]
We also let
\[
    \vartheta:
    \left\{
        I+A\in\SL_{m+1}(\R):
        \sum_{i,j}|A_{ij}|\leq \varepsilon/10
    \right\}
    \longrightarrow \R_{\geq 0}
\]
be a smooth bump function, extended by zero to $G$, satisfying $m_G(\vartheta)=1$, where $m_G$
denotes Haar measure on $G$, normalized so that a fundamental domain
for the action of $\Gamma$ on $G$ has measure $1$, and
\[
    \|\vartheta\|_{C^{l_0}}\leq \varepsilon^{-B'},
\]
where $B'=B'(G,l_0)>0$.

With this notation, define
\[
    \varphi_k
    =
    \vartheta*(\chi_k\widehat{\ind}_{R_k^-}).
\]
Then $\varphi_k$ is a smooth bounded function on $\X$ satisfying
\[
    0\leq \varphi_k\leq \widehat{\ind}_{R_k}.
\]
Consequently,
\[
    \rS_k(\theta)
    \geq
    \varphi_k(a_ku(\theta)\Gamma).
\]
Thus, it suffices to obtain a lower bound for
\[
    \sum_{k\in K_{\mathrm{small}}}
    \varphi_k(a_ku(\theta)\Gamma).
\]

The proof is essentially identical to that of~\cite[Prop.~9.12]{benard2026};
we briefly indicate the required modifications. In the order fixed above,
we choose $\gamma_2>0$ sufficiently small and then $\gamma_1$ sufficiently
small compared with $\gamma_2$.

Under these choices, the estimates for
$\mu_\X(\varphi_k)$, $\|\varphi_k\|_{C^{l_0}}$, and
$\mathcal{S}_{2,l_0}(\varphi_k)$ obtained in~\cite[Lem.~9.10]{benard2026}
continue to hold in the present setting, using the same arguments
together with Lemma~\ref{lem:imp change}.

For $k\in K_{\mathrm{small}}$, define the random variables
\[
    Y_k:\R^m\to\R,
    \qquad
    Y_k(\theta)
    =
    \varphi_k(a_ku(\theta)\Gamma),
\]
and set
\[
    y_k=\mu_\X(\varphi_k),
    \qquad
    Z_k=Y_k-y_k.
\]
By arguments as in~\cite[Lem.~9.11]{benard2026},
\[
    \mathbb{E}(Z_k^2)\ll y_k,
\]
provided that $y_k\geq1$. Combining this estimate with
Lemma~\ref{lem:imp change} and arguing as in
\cite[Prop.~9.12]{benard2026}, whose hypothesis
\cite[(9.18)]{benard2026} holds for every $k\in K_{\mathrm{small}}$ by the
definition of $K_{\mathrm{tiny}}$, we obtain that, for every finite subset
$J\subset K_{\mathrm{small}}$ satisfying
$\inf J\gg_{\gamma_2}1$,
\begin{align}
\label{eq:exp K small}
    \mathbb{E}\left[
        \left(\sum_{j\in J}Z_j\right)^2
    \right]
    \ll
    \left(
        1+\sum_{j\in J}y_j
    \right)^{3/2}.
\end{align}

Then, by the abstract probabilistic result
\cite[Lem.~9.13]{benard2026}, together with~\eqref{eq:exp K small}
and~\eqref{eq:assum K small div},
for $\nu$-almost every $\theta$ and all sufficiently large $n$,
\[
    \sum_{\substack{k\in K_{\mathrm{small}}\\ k\leq n}}
    \rS_k(\theta)
    \geq
    \left(1-\tau^{-1}-C_m\varepsilon\right)
    \sum_{\substack{k\in K_{\mathrm{small}}\\ k\leq n}}
    \psi(\tau^k)\tau^k,
\]
where $C_m>0$ depends only on $m$. Choosing $\varepsilon$ so that
$C_m\varepsilon\leq\eta$ gives~\eqref{eq:estimate K small}.
\end{proof}

\subsubsection{Conclusion}

We may ignore $K_{\mathrm{tiny}}$ since, by the definition of
$K_{\mathrm{tiny}}$,
\[
    \sum_{k\in K_{\mathrm{tiny}}}\psi(\tau^k)\tau^k<\infty.
\]
Indeed, write $K_{\mathrm{tiny}}=K'\cup K''$, where $K'$ consists of those $k$ satisfying the
first condition in~\eqref{eq:assum} and $K''$ of those satisfying the second.
If $k\in K''$, then
$\psi(\tau^k)\tau^k<\log^{-2}(\tau^k)=(k\log\tau)^{-2}$, and the corresponding series
converges. As for $K'$, recall from Remark~\ref{rem:balanced} that $\alpha\leq1/m$, so that
$2-m\alpha\geq1>0$; since $\psi(\tau^k)<1$ and $\tau^k>1$, the first condition
in~\eqref{eq:assum} can hold only if $2-(2m+1)\alpha>0$, in which case it gives
\[
    \psi(\tau^k)<(\tau^k)^{-\frac{2-m\alpha}{2-(2m+1)\alpha}},
    \qquad\text{and hence}\qquad
    \psi(\tau^k)\tau^k<(\tau^k)^{-\frac{(m+1)\alpha}{2-(2m+1)\alpha}} .
\]
As the exponent on the right is negative, the series over $K'$ converges as well.

Similarly, we may ignore $K_{\mathrm{big}}$ or $K_{\mathrm{small}}$
whenever the corresponding sum
\[
    \sum_{k\in K_{\mathrm{big}}}\psi(\tau^k)\tau^k
    \qquad\text{or}\qquad
    \sum_{k\in K_{\mathrm{small}}}\psi(\tau^k)\tau^k
\]
is finite.
For every remaining set, the required lower bound follows from
Lemmas~\ref{lem:lower K big} and~\ref{lem:lower K small}. Since
\[
    \sum_{k\in\N}\psi(\tau^k)\tau^k=\infty,
\]
at least one of the two sums above is infinite, the sets with a
convergent sum contribute $O(1)$, and Lemmas~\ref{lem:lower K big}
and~\ref{lem:lower K small} bound the rest from below by
$(1-\tau^{-1}-\eta)\sum_{k\leq n}\psi(\tau^k)\tau^k$ up to $O(1)$. Hence
\eqref{eq:red 1 div case}
follows, which completes the proof of the lower bound.

\subsection{Upper bound}
\label{subsec:div upper}

The proof of the upper bound follows along the same lines as that of the
lower bound. Namely, for $k\in\N$ and $\tau>1$, define
\[
    \rS_{\tau,k}^+(\theta)
    = \#\left\{
(p_1,\ldots,p_m,q)\in\Z^m\times\N:
\begin{array}{l}
q\in[\tau^{k},\tau^{k+1}],\\
0<p_i+\theta_iq<\psi_i(\tau^k)
\text{ for all }1\leq i\leq m
\end{array}
\right\}.
\]
Here the range of $q$ has been shifted by one and left closed at both
ends: for $q\geq\tau^k$ one has $\psi_i(q)\leq\psi_i(\tau^k)$, and the
overlap of consecutive ranges only weakens the upper bound.
As in the lower bound, it is enough to prove that for every
$\tau\in(1,2]$ and for $\nu$-almost every $\theta\in\R^m$, for every
$\eta>0$ and all sufficiently large $n$,
\begin{align}
\label{eq:red 2 div case}
    \sum_{k=1}^n \rS_{\tau,k}^+(\theta)
    \leq
    (\tau-1+\eta)
    \sum_{k=1}^n\psi(\tau^k)\tau^k.
\end{align}
Indeed, let $\tau^n\leq T<\tau^{n+1}$. Since $\psi_i$ is
non-increasing, every pair counted by $\cN^+_T(\theta,\Psi)$
with $q\geq1$ lies in some range $q\in[\tau^k,\tau^{k+1}]$ with $k\leq n$ and
satisfies  $0<p_i+\theta_iq<\psi_i(q)\leq\psi_i(\tau^k)$, so that
\[
    \cN^+_T(\theta,\Psi)\ \leq\ \sum_{k=0}^n\rS^+_{\tau,k}(\theta)
    \ =\ \rS^+_{\tau,0}(\theta)+\sum_{k=1}^n\rS^+_{\tau,k}(\theta),
\]
the term $\rS^+_{\tau,0}(\theta)$ being finite and independent of $T$, hence
negligible compared with $S_\Psi(T)$.
Moreover, with $\varepsilon\in(0,\tfrac12)$ and $k\geq k_0(\tau,\varepsilon)$ as in the proof of
Lemma~\ref{lem:enough for power series},
\[
    \sum_{q=\lceil\tau^{k-1}\rceil}^{\lceil\tau^{k}\rceil-1}\!\!\psi(q)
    \ \geq\
    \bigl(\lceil\tau^{k}\rceil-\lceil\tau^{k-1}\rceil\bigr)\psi(\tau^k)
    \ \geq\
    (1-\varepsilon)(1-\tau^{-1})\,\psi(\tau^k)\tau^k ,
\]
whence $(1-\varepsilon)(1-\tau^{-1})\sum_{k\leq n}\psi(\tau^k)\tau^k\leq(1+\varepsilon)S_\Psi(T)$ for
$T$ large. Together with~\eqref{eq:red 2 div case} this gives
\[
    \limsup_{T\to\infty}\frac{\cN^+_T(\theta,\Psi)}{S_\Psi(T)}
    \ \leq\
    \frac{(1+\varepsilon)(\tau-1+\eta)}{(1-\varepsilon)(1-\tau^{-1})}
    \ =\
    \frac{(1+\varepsilon)\,\tau}{1-\varepsilon}\Bigl(1+\frac{\eta}{\tau-1}\Bigr),
\]
and letting $\varepsilon\to0$, then $\eta\to0$, then $\tau\to1^+$ along countable
sequences as before yields
\[
    \limsup_{T\to\infty}\frac{\cN^+_T(\theta,\Psi)}{S_\Psi(T)}\leq1
\]
for $\nu$-almost every $\theta$. Thus~\eqref{eq:red 1 div case} and~\eqref{eq:red 2 div case}
together imply~\eqref{eq:one sided asym}, which
by~\eqref{eq:orthant decomposition} gives Proposition~\ref{prop:div case}.
We now prove~\eqref{eq:red 2 div case}.

We fix $\tau>1$ and, as before, write
$\rS_{\tau,k}^+(\theta)=\rS_k^+(\theta)$. We note that
\[
    \rS_k^+(\theta)
    =
    \widehat{\ind}_{P_k}(a_ku(\theta)\Gamma),
\]
where \[
    P_k=[0,r_k]^m\times[r_k,\tau r_k],
\]
and $r_k$ and $a_k$ are defined in
\eqref{eq:def r k} and \eqref{eq:def a k}, respectively.

We again fix $\gamma_1\in(0,1/2)$ as above and partition $\N$ into
$K_{\mathrm{tiny}}$, $K_{\mathrm{big}}$, and $K_{\mathrm{small}}$.

\subsubsection{Upper asymptotics over $K_{\mathrm{big}}$}

The argument is identical to that for the lower bound. Namely, using the
upper bound in \cite[Lem.~9.6]{benard2026} in place of its lower bound,
and proceeding exactly as in the proof of
\eqref{eq:sum over K big}, we obtain that for every $\eta>0$,
\[
    \sum_{\substack{k\in K_{\mathrm{big}}\\k\leq n}}
    \rS_k^+(\theta)
    \leq
    (\tau-1+\eta)
    \sum_{\substack{k\in K_{\mathrm{big}}\\k\leq n}}
    \psi(\tau^k)\tau^k
\]
for $\nu$-almost every $\theta$ and all sufficiently large $n$, provided that
\[
\sum_{k\in K_{\mathrm{big}}} \psi(\tau^k)\tau^k= \infty.
\]

\subsubsection{Upper asymptotics over $K_{\mathrm{small}}$}

We fix $\gamma_2,\varepsilon>0$ sufficiently small and let
$\chi_k$ and $\vartheta$ be as in the proof of the lower bound. Define
\[
    P_k^+
    =
    [-\varepsilon r_k,(1+\varepsilon)r_k]^m
    \times[(1-\varepsilon)r_k,(\tau+\varepsilon)r_k]
\]
and
\[
    \varphi_k^+
    =
    \vartheta*(\chi_k\widehat{\ind}_{P_k^+}).
\]
The truncation by $\chi_k$ requires an additional argument when
comparing $\varphi_k^+$ with $\widehat{\ind}_{P_k}$. By
Lemma~\ref{lem:dist is log}, for $\nu$-almost every $\theta$ and all
sufficiently large $k$,
\[
    \operatorname{dist}(a_ku(\theta)\Gamma,\Gamma)
    \ll \log\log t_k.
\]
Since the injectivity radius is log-Lipschitz, this implies, for some
$C_0=C_0(m)>0$,
\[
    \inj(g^{-1}a_ku(\theta)\Gamma) \gg \inj(a_ku(\theta)\Gamma)
    \gg (\log t_k)^{-C_0}
    \gg t_k^{-\gamma_2}
\]
for all sufficiently large $k$ and $g \in \supp(\vartheta)$. Hence
\[
    \chi_k(g^{-1}a_ku(\theta)\Gamma)=1
\]
for all sufficiently large $k$ and $g \in \supp(\vartheta)$.

Consequently, for $\nu$-almost every $\theta$ and all sufficiently
large $k$,
\[
    \varphi_k^+(a_ku(\theta)\Gamma)
    =
    \bigl(\vartheta*\widehat{\ind}_{P_k^+}\bigr)
    (a_ku(\theta)\Gamma).
\]
By the choice of $P_k^+$ and $\vartheta$, we have
\[
    \vartheta*\widehat{\ind}_{P_k^+}
    \geq
    \widehat{\ind}_{P_k}.
\]
Thus, for $\nu$-almost every $\theta$ and all sufficiently large $k$,
\[
    \varphi_k^+(a_ku(\theta)\Gamma)
    \geq
    \widehat{\ind}_{P_k}(a_ku(\theta)\Gamma)
    =
    \rS_k^+(\theta).
\]
The estimates for the mean, Sobolev norms, and correlations of
$\varphi_k^+$ are obtained exactly as in the lower-bound argument.
Consequently, the argument proving
Lemma~\ref{lem:lower K small} applies verbatim and gives
\[
    \sum_{\substack{k\in K_{\mathrm{small}}\\k\leq n}}
    \rS_k^+(\theta)
    \leq
    (\tau-1+\eta)
    \sum_{\substack{k\in K_{\mathrm{small}}\\k\leq n}}
    \psi(\tau^k)\tau^k
\]
for $\nu$-almost every $\theta$ and all sufficiently large $n$.

\subsubsection{Conclusion}

As in the proof of the lower bound, we may ignore
$K_{\mathrm{tiny}}$, as well as $K_{\mathrm{big}}$ or
$K_{\mathrm{small}}$ whenever the corresponding sum
\[
    \sum_{k\in J}\psi(\tau^k)\tau^k
\]
is finite. More precisely, for every fixed
$J\subset\N$ satisfying
\[
    \sum_{j\in J}\psi(\tau^j)\tau^j<\infty,
\]
we have, using arguments as in the proof of
Proposition~\ref{prop:convergence case}, whose hypothesis~\eqref{eq:e5} is
assumed here, that for $\nu$-almost every $\theta$,
\[
    \sum_{j\in J}\rS_{\tau,j}^+(\theta)<\infty.
\]

For each of the remaining sets, the preceding estimates give the
required upper bound. Since
\[
    \sum_{k\in\N}\psi(\tau^k)\tau^k=\infty,
\]
the contribution from the sets for which the corresponding series
converges is negligible. Hence, for $\nu$-almost every $\theta$ and all
sufficiently large $n$,
\[
    \sum_{k=1}^n\rS_k^+(\theta)
    \leq
    (\tau-1+\eta)
    \sum_{k=1}^n\psi(\tau^k)\tau^k.
\]
This proves \eqref{eq:red 2 div case} and, hence, the upper bound. By the reduction carried out at the beginning of
Section~\ref{sec:div case}, this establishes~\eqref{eq:one sided asym} for
$\boldsymbol{\omega}=(1,\ldots,1)$, and hence, by the remark
following~\eqref{eq:one sided asym}, for every
$\boldsymbol{\omega}\in\{\pm1\}^m$. Since
\eqref{eq:orthant decomposition} holds for $\nu$-almost every $\theta$, we
conclude that
$\cN_T(\theta,\Psi)\sim2^m\sum_{q\leq T}\psi(q)$ for $\nu$-almost every
$\theta$. This completes the proof of Proposition~\ref{prop:div case}.
\hfill\qedsymbol

\section{Proof of the counting theorem}
\label{sec:proof of main}

\begin{proof}[Proof of Theorem~\ref{main thm:counting}]
We first check that $\nu=\nu_1\otimes\cdots\otimes\nu_m$ is self-similar on $\R^m$
in the sense of~\cite{benard2026}. Write
$\nu_i=\sum_{j\in E_i}\lambda_{i,j}(\phi_{i,j})_*\nu_i$ with
$\phi_{i,j}(x)=cx+y_{i,j}$, and for
$\mathbf j=(j_1,\ldots,j_m)\in E_1\times\cdots\times E_m$ put
$\lambda_{\mathbf j}=\lambda_{1,j_1}\cdots\lambda_{m,j_m}$ and
$\Phi_{\mathbf j}(x)=c\,\mathrm{Id}\,x+(y_{1,j_1},\ldots,y_{m,j_m})$. Because
the $m$ contraction ratios are equal to one and the same $c$, each
$\Phi_{\mathbf j}$ is a \emph{similarity} of $\R^m$ --- this is where that
hypothesis is used --- and
$(\Phi_{\mathbf j})_*\nu=\bigotimes_{k=1}^m(\phi_{k,j_k})_*\nu_k$ shows that
$\nu$ is stationary under $\mathcal L=\sum_{\mathbf j}\lambda_{\mathbf j}
\delta_{\Phi_{\mathbf j}}$, which has a finite exponential moment since it is
finitely supported. It is irreducible: for each $i$ the maps $\phi_{i,j}$ have
no common fixed point, so $\nu_i$ is not a Dirac mass and
$\supp(\nu)=\supp(\nu_1)\times\cdots\times\supp(\nu_m)$ contains the $2^m$
vertices of a nondegenerate box, hence affinely spans $\R^m$, while any affine
subspace invariant under all $\Phi_{\mathbf j}$ contains $\supp(\nu)$.

By \cite[Thm.~1.2]{benard2026}, the measure $\nu$ satisfies the hypotheses of
Theorem~\ref{thm:Double equi} for
\[
    \bfw=(1/m,\ldots,1/m).
\]
Indeed, the diagonal flow of \cite[Thm.~1.2]{benard2026} is
$a(T)=\operatorname{diag}\bigl(T^{\frac{1}{m+1}},\ldots,T^{\frac{1}{m+1}},
T^{-\frac{m}{m+1}}\bigr)$, and the unipotent $u(\theta)$ is the same in both
papers, so that
\[
    g_t(1/m,\ldots,1/m)=a\bigl(e^{t(m+1)/m}\bigr).
\]
The norm occurring in the error term of
\cite[Thm.~1.2]{benard2026} is the $C^{l_0}$-norm with
$l_0=\bigl\lceil\tfrac12\dim SO(m+1)\bigr\rceil$. Substituting
$T=e^{t(m+1)/m}$ in the bound $C\,\inj(x)^{-1}T^{-c}\|f\|_{C^{l_0}}$ therefore
turns it into~\eqref{eq:e3}, with
\[
    \delta_{\bfw}=\frac{(m+1)c}{m},
    \qquad
    c_{\bfw}=1 .
\]
 Therefore Theorem~\ref{thm:Double equi}, applied with $\alpha/2$ in place
of $\alpha$, yields the second hypothesis of Proposition~\ref{prop:div case},
while Proposition~\ref{prop:change of weights}, applied at $x=\Gamma$, yields
\eqref{eq:e5} for every weight in $[\alpha/2,1]^m$ and hence the remaining
hypotheses of Propositions~\ref{prop:convergence case}
and~\ref{prop:div case}.
Finally, for each $i$ the maps $\phi_{i,j}$ have no common fixed
point, so $\nu_i$ is non-atomic by Lemma~\ref{lem:frostman}; since
$\nu(\{\theta_i=c\})=\nu_i(\{c\})$, the coordinate projections of $\nu$ are
non-atomic, which is the last hypothesis of
Proposition~\ref{prop:div case}.
Applying these two propositions gives the convergence and divergence
cases of Theorem~\ref{main thm:counting}, respectively. This completes
the proof.
\end{proof}

\bibliographystyle{amsplain}
\bibliography{Biblio}

\end{document}